\documentclass[11pt,a4paper,reqno]{amsart}
\usepackage{amsfonts}
\usepackage{booktabs}
\usepackage{amsmath,amssymb,amsthm,amsxtra}
\usepackage{float}
\usepackage{amssymb}
\usepackage{mathabx}
\usepackage{bbm}
\usepackage[colorlinks, linkcolor=blue,anchorcolor=Periwinkle,
citecolor=Red,urlcolor=Emerald]{hyperref}
\usepackage[usenames,dvipsnames]{xcolor}
\usepackage{enumitem}
\usepackage{geometry,array} 
\usepackage{graphicx}
\usepackage{subfigure}
\usepackage{bookmark}
\usepackage{tikz}\usetikzlibrary{matrix}
\usepackage{url}
\usepackage{dsfont}
\usepackage[colorinlistoftodos]{todonotes}
\usepackage{tablists} \restorelistitem

\usetikzlibrary{decorations.markings}
\tikzset{->-/.style={decoration={  markings,  mark=at position #1 with
    {\arrow{>}}},postaction={decorate}}}
\tikzset{-<-/.style={decoration={  markings,  mark=at position #1 with
    {\arrow{<}}},postaction={decorate}}}
\usepackage{extarrows}
\usepackage[all]{xy}
\usepackage{setspace}
\usepackage{thmtools}
\usepackage{thm-restate}
\usepackage{hyperref}
\usepackage{cleveref}
\usepackage{multirow}

\newcommand{\sgn}{\operatorname{sgn}}

\renewcommand{\mod}{\operatorname{mod}}

\newcommand{\ise}{\operatorname{ise}}

\newcommand{\mfx}{\mathbf{x}}

\newcommand{\mcA}{\mathcal{A}}

\newcommand{\mcF}{\mathcal{F}}

\newcommand{\mcL}{\mathcal{L}}

\newcommand{\mcT}{\mathcal{T}}
\newcommand{\mcU}{\mathcal{U}}

\newcommand{\mbN}{\mathbb{N}}

\newcommand{\mbQ}{\mathbb{Q}}
\newcommand{\mbR}{\mathbb{R}}

\newcommand{\mbT}{\mathbb{T}}

\newcommand{\mbZ}{\mathbb{Z}}

\theoremstyle{plain}
\newtheorem{theorem}{Theorem}[section]

\newtheorem{lemma}[theorem]{Lemma}
\newtheorem{corollary}[theorem]{Corollary}
\newtheorem{proposition}[theorem]{Proposition}

\theoremstyle{definition}
\newtheorem{definition}[theorem]{Definition}

\newtheorem{example}[theorem]{Example}

\newtheorem{remark}[theorem]{Remark}

\numberwithin{equation}{section}
\newtheorem{definition-proposition}[theorem]{Definition-Proposition}

\begin{document}

\title {A cluster theory approach from mutation invariants to Diophantine equations}

\date{\today}

\author{Zhichao Chen}
\address{School of Mathematical Sciences\\ University of Science and Technology of China \\ Hefei, Anhui 230026, P. R. China. \& Graduate School of Mathematics\\ nagoya University\\Chikusa-ku\\ Nagoya\\464-8602\\ Japan.}
\email{czc98@mail.ustc.edu.cn}
\author{Zixu Li}
\address{Lizx: Yancheng Institute of Technology,
    224051, YanCheng, Jiangsu,
    China}
\email{lizx19@tsinghua.org.cn}
\maketitle

\begin{abstract}
In this paper, we define and classify the sign-equivalent exchange matrices. We give a Diophantine explanation for the differences between rank 2 cluster algebras of finite type and affine type based on \cite{CL24}. We classify the positive integer points of the Markov mutation invariant and its variant. As an application, several classes of Diophantine equations with cluster algebraic structures are exhibited.  
\end{abstract}
\tableofcontents
\section{Introduction}
Cluster algebras are important commutative algebras with different generators and relations. They were first introduced in \cite{FZ02, FZ03} to investigate the total positivity of Lie groups and canonical bases of quantum groups. Nowadays, cluster algebras are closely related to different subjects in mathematics, such as  representation theory \cite{BMRRT06, BIRS09, HL10, KY11}, higher Teichm{\"u}ller theory \cite{FG09}, integrable system \cite{KNS11}, Poisson geometry \cite{GSV12},
	 commutative algebras \cite{Mul13, Mul14}, combinatorics \cite{FR05, NZ12, LLZ14, RS18, GHKK18, FG19} and number theory \cite{Pro20, BBH11, PZ12, Lam16, LLRS23, Hua22, GM23, BL24, CL24}.

The relations between cluster algebras and Diophantine equations were firstly discovered by Propp \cite{Pro20}. Later, Beineke-Br{\"u}stle-Hille \cite{BBH11}, Peng-Zhang \cite{PZ12}, Lee-Li-Rabideau-Schiffler \cite{LLRS23} and Huang \cite{Hua22} studied some properties and conjectures about the famous Markov equation
\begin{align}
	x_1^2+x_2^2+x_3^2=3x_1x_2x_3,\label{markov 0}
\end{align} whose all solutions can one-to-one correspond to all clusters of the once-punctured torus cluster algebra. On this basis, 
Gyoda-Matsushita \cite{GM23} solved the generalized Markov equations
\begin{align}
x_1^2+x_2^2+x_3^2+k_1x_1x_2+k_2x_2x_3+k_3x_1x_3=(3+k_1+k_2+k_3)x_1x_2x_3\label{GM-equation 0}
\end{align} and studied their structures of generalized cluster algebras.
In addition,
Lampe \cite{Lam16} proved that all the solutions
to a variant of Markov Diophantine equation 
\begin{align}
x_1^2+x_2^4+x_3^4+2x_1x_2^2+2x_1x_3^2=7x_1x_2^2x_3^2\label{lampe-equation 0}
\end{align} can be generated by the initial solution $(1,1,1)$ through finite cluster mutations. Afterwards, Bao-Li \cite{BL24} followed his work and gave a criterion to determine  the solutions in the orbit of the initial solution under the actions of a certain group. Then, Chen-Li \cite{CL24} defined the mutation invariants and classified finite type mutation invariants of rank 2 as follows:
\begin{align}\label{main formula 0}
		\mcT(x_1,x_2)=&{}\Phi(F(c_{1;1}(x_1,x_2),c_{2;1}(x_1,x_2)),\cdots,F(c_{1;m}(x_1,x_2),c_{2;m}(x_1,x_2))),
	\end{align} where
	$(c_{1;i}(x_1,x_2),c_{2;i}(x_1,x_2))^m_{i=1}$ are $m$ distinct clusters, $\Phi(X_1,\cdots,X_m)$ is a symmetric polynomial with $m$ variables over $\mbQ$ and $F(X_1,X_2)$ is a rational function.

In this paper, we study the properties of diophantine equations corresponding to the mutation invariants \eqref{main formula 0} and exhibit the relation between them.
Firstly, to find more mutation invariants of cluster algebras, it is crucial to classify the sign-equivalent exchange matrices since their mutation rules are invariant, see \Cref{def sign} and \Cref{IMR rules}. It is a special class of finite mutation type, see \Cref{def finite mutation type}. The cluster algebras of finite mutation type were studied via  block-decomposable quivers and unfoldings in \cite{FST1, FST2}. Here, we calculate and classify all the irreducible sign-equivalent exchange matrices by cluster mutation rules as the following.
\begin{theorem}[\Cref{main}]
	All the irreducible sign-equivalent exchange matrices are as \Cref{ta1}.
\end{theorem}
\begin{table}[ht!] \caption{A list of all the irreducible sign-equivalent exchange matrices}
\renewcommand{\arraystretch}{0.82}
\begin{center}
		\scalebox{0.85}{\begin{tabular}{|c|c|}
						\hline  
			{\begin{tabular}{c}\\ Rank $n$ \\
					\vspace{0.01mm}
			\end{tabular}} & The irreducible sign-equivalent exchange matrices  \\
			\hline 
			$n=1$ & {\begin{tabular}{c} \\ None. 
				\\	 \vspace{0.01mm} 
			\end{tabular}} \\
			\hline
			$n=2$ & {\begin{tabular}{c}\\ $\sigma\begin{pmatrix}
	0 & b \\ -c & 0 \end{pmatrix}$, where $b,c\in \mbN_+$ and $\sigma\in \mathfrak{S}_2$. \\
					\vspace{0.01mm}
			\end{tabular}} \\
			\hline

			\multirow{11}{*}{$n=3$} & {\begin{tabular}{c}\\ $\sigma\begin{pmatrix}0 & 2 & -2\\ -2 & 0 & 2\\ 2 & -2 & 0\end{pmatrix}$, where $\sigma\in \mathfrak{S}_3$. \\
					\vspace{0.01mm}
			\end{tabular}} \\
			\cline{2-2}
			& {\begin{tabular}{c}\\ $\sigma\begin{pmatrix}0 & 1 & -1\\ -4 & 0 & 2\\ 4 & -2 & 0\end{pmatrix}$, where $\sigma\in \mathfrak{S}_3$. \\
					\vspace{0.01mm}
			\end{tabular}} \\
			\cline{2-2} & {\begin{tabular}{c}\\ $\sigma\begin{pmatrix}0 & -4 & 4\\ 1 & 0 & -2\\ -1 & 2 & 0\end{pmatrix}$, where $\sigma\in \mathfrak{S}_3$. \\
					\vspace{0.01mm}
			\end{tabular}}  \\
			\hline 
						$n\geq 4$ & {\begin{tabular}{c} \\ None. 
					 \\ \vspace{0.01mm}
			\end{tabular}} \\
			\hline 
		\end{tabular}}\\ \
		
	\end{center} \label{ta1}
\end{table}

For the rank 2 cluster algebras, based on the formula \eqref{main formula 0} and motivated by the Laurent phenomenon and positivity of cluster algebras, we focus on the cases that $\Phi(u_1,\cdots,u_m)\in \mbQ_{\geq 0}[u_1,\cdots,u_m]$ and $F(X_1,X_2)\in \mbQ_{\geq 0}[X_1,X_2]$. Then, the corresponding Diophantine equations are 
\begin{align}
	\mcT(x_1,x_2)=\mcT(a,b), \label{finite equation 0}
\end{align} where $(a,b)\in \mbN_+^2$,
and we prove the finiteness of the solutions. We also provide a Diophantine explanation for the differences between finite type and affine type cluster algebras of rank 2, see \Cref{diophantine explanation}.
\begin{theorem}[\Cref{finite solution}]
There are only finite positive integer solutions to the \Cref{finite equation 0}.
\end{theorem}
For the rank 3 cluster algebras, recall that Lampe \cite{Lam16} exhibited two Laurent mutation invariants of $\mcA_{P}$ and $\mcA_{L}$ as follows (also see \eqref{2-matrix} and \eqref{4-matrix}): 
\begin{align}
	\mcT_1(x_1,x_2,x_3)=\dfrac{x_1^2+x_2^2+x_3^2}{x_1x_2x_3},\label{Lampe mutation invariant 0}
\end{align}
\begin{align}
\mcT_2(x_1,x_2,x_3)=\dfrac{x_1^2+x_2^4+x_3^4+2x_1x_2^2+2x_1x_3^2}{x_1x_2^2x_3^2}.\label{Lampe mutation invariant 1}\end{align}
A natural question is to find all the positive integer points of the mutation invariants, that is $(x_1,x_2,x_3,\mcT_i(x_1,x_2,x_3))\in \mbN_+^4$. In fact, it is equivalent to solve the Diophantine equations $\mcT_{i}(x_1,x_2,x_3)=k$, where $i=1,2$ and $k\in \mbN$. Note that Aigner found all the positive integer points of \eqref{Lampe mutation invariant 0} by use of the fact that three components in a Markov triple are pairwise relatively prime, cf. \cite[Proposition 2.2]{Aig13}. However, we will give a different proof by use of cluster mutations.
\begin{theorem}[Aigner, \Cref{Aigner's theorem}]
	Let $k\in \mbN$ and the Diophantine equation be as follows:
	\begin{align}
		x_1^2+x_2^2+x_3^2=kx_1x_2x_3. \label{Markov's equation 0}
	\end{align}
	Then, it has positive integer solutions if and only if $k=1$ or $k=3$.
\end{theorem}
Furthermore, we find all the positive integer points of \eqref{Lampe mutation invariant 1} as follows by cluster mutations.
\begin{theorem}[\Cref{analogue}]
	Let $k\in \mbN$ and the Diophantine equation be as follows:
	\begin{align}
x_1^2+x_2^4+x_3^4+2x_1x_2^2+2x_1x_3^2=kx_1x_2^2x_3^2. \label{Lampe's equation 0}
	\end{align}
	Then, it has positive integer solutions if and only if $k=7$.
\end{theorem}
As an application, we get two classes of Diophantine equations with cluster algebraic structures, that is, all their positive integer solutions can be generated by the initial solution $(1,1,1)$ through finite cluster mutations. The results generalize the results about two isolated Diophantine equations given by \cite[Theorem 2.3 \& Theorem2.6]{Lam16}.
\begin{proposition}[\Cref{Prop case}]
	Let $F(X)\in \mbZ[X]$ be a non-constant monic polynomial and $t\in \mbZ$. The Diophantine equation \begin{align}
		F\left(\dfrac{x_1^2+x_2^2+x_3^2}{x_1x_2x_3}\right)=F(t) \label{F222 0}
	\end{align} has positive integer solutions if and only if $F(t)=F(1)$ or $F(t)=F(3)$. Moreover, \begin{enumerate}[leftmargin=2em]
		\item if $F(1)=F(3)$, all the positive integer solutions can be obtained from the initial solutions $(1,1,1)$ and $(3,3,3)$ by finite cluster mutations of $\mcA_P$;
		\item if $F(1)\neq  F(3)$, all the positive integer solutions can be obtained from either the initial solution $(1,1,1)$ or the initial solution $(3,3,3)$ by finite cluster mutations of $\mcA_P$.
	\end{enumerate} 
\end{proposition}

\begin{proposition}[\Cref{Lampe case}]
	Let $F(X)\in \mbZ[X]$ be a non-constant monic polynomial and $t\in \mbZ$. The Diophantine equation \begin{align}
F\left(\dfrac{x_1^2+x_2^4+x_3^4+2x_1x_2^2+2x_1x_3^2}{x_1x_2^2x_3^2}\right)=F(t) \label{F124 0}
	\end{align} has positive integer solutions if and only if $F(t)=F(7)$. Under this case, all the positive integer solutions can be obtained from the initial solution $(1,1,1)$ by finite cluster mutations of $\mcA_L$.
\end{proposition}

The paper is organized as follows. In \Cref{S2}, we review the basic notions and properties about cluster algebras and mutation invariants. We also define sign-equivalent exchange matrices, see \Cref{def sign}. In \Cref{S3}, we calculate and classify all the irreducible sign-equivalent exchange matrices by cluster mutations, see \Cref{ta1} and \Cref{main}. In \Cref{S4}, we prove the finiteness of the solutions to Diophantine equations corresponding to rank 2 mutation invariants of finite type, see \Cref{finite solution}. Then, we give a Diophantine explanation for the differences between finite type and affine type cluster algebras of rank 2, see \Cref{diophantine explanation}. In \Cref{main section}, we find all the positive integer points of the Markov mutation invariant \eqref{Lampe mutation invariant 0} and its variant \eqref{Lampe mutation invariant 1} by use of cluster mutations, see \Cref{Aigner's theorem} and \Cref{analogue}. In \Cref{S6}, as an application, we find two classes of Diophantine equations with cluster algebraic structures, see \Cref{Prop case} and \Cref{Lampe case}. 
\section*{Conventions}
\begin{itemize}[leftmargin=1em]\itemsep=0pt
\item In this paper, the integer ring, the set of non-negative integers, the set of positive integers, the rational number field and the real number field are denoted by $\mbZ$, $\mbN$, $\mbN_{+}$, $\mbQ$ and $\mbR$ respectively. 
\item We denote by $\text{Mat}_{n\times n}(\mbZ)$ the set of all $n\times n$ integer square matrices. An integer square matrix $B$ is said to be \emph{skew-symmetrizable} if there exists a positive integer diagonal matrix $D$ such that $DB$ is skew-symmetric and $D$ is called the \emph{left skew-symmetrizer} of $B$. 
\item For any $a\in \mbZ$, we denote $[a]_{+}=\max(a,0)$ and then $a=[a]_+-[-a]_+$. The zero matrix is denoted by $O$.
\item Let $[\ \cdot \ ]$ be the floor function over $\mbR$. That is, for any $a\in \mbR$, $[a]=\max\{x\in \mbZ|\ x\leq a\}$.
\end{itemize}
\section{Preliminaries}\label{S2}
In this section, we recall some basic notions and properties about cluster algebras introduced by Fomin-Zelevinsky \cite{FZ02,FZ03,FZ07,Nak23}. We also recall the notions of mutation invariants of cluster algebras defined in \cite{CL24}.
\subsection{Seeds and cluster algebras}\

Let $n\in \mbN_{+}$ and $\mcF$ be a rational function field of $n$ variables over $\mbQ$. A \emph{seed} is a pair $(\mfx,B)$, such that $\mfx=(x_1,\dots,x_n)$ is an $n$-tuple of algebraically independent and generating elements of $\mcF$ and $B=(b_{ij})_{n\times n}$ is a skew-symmetrizable matrix. The $n$-tuple $\mfx$ is the \emph{cluster}, elements $x_i\in \mfx$ are \emph{cluster variables} and $B$ is the \emph{exchange matrix}.

Given a seed $(\mfx,B)$ and $1\leq k\leq n$, we define $\mu_k(\mfx,B)=(\mfx^{\prime},B^{\prime})$ such that $\mfx^{\prime}=(x_{1}^{\prime},\dots,x_{n}^{\prime})$, where \begin{align}\ 
		x_{i}^{\prime}=\left\{
		\begin{array}{ll}
			x_{k}^{-1}(\prod\limits_{j=1}^{n}x_{j}^{[b_{jk}]_{+}}+\prod\limits_{j=1}^{n}x_{j}^{[-b_{jk}]_{+}}), &   \text{if}\ i=k, \\
			x_{i}, &  \text{if}\ i \neq k, 
		\end{array} \right.
	\end{align}
and $B^{\prime}=(b_{ij}^{\prime})_{n\times n}$ is given by 
	\begin{align} \label{matrix mutation}
		b_{ij}^{\prime}=\left\{
		\begin{array}{ll}
			-b_{ij}, &  \text{if}\ i=k \;\;\mbox{or}\;\; j=k, \\
			b_{ij}+[b_{ik}]_{+}b_{kj}+b_{ik}[-b_{kj}]_{+}, &  \text{if}\ i\neq k \;\;
			\mbox{and}\; j\neq k. 
		\end{array} \right. 
	\end{align}
In fact, $(\mfx^{\prime},B^{\prime})$ is still a seed and $\mu_k$ is involutive, that is $\mu_k(\mfx^{\prime},B^{\prime})=(\mfx,B)$, see \cite{Nak23}. Then, $(\mfx^{\prime},B^{\prime})$ is called the \emph{k-direction mutation} of $(\mfx,B)$. Two seeds (exchange matrices) are said to be \emph{mutation-equivalent} if one can be obtained from the other by a sequence of mutations. The mutation equivalence class of $B$ is denoted by $[B]$. A \emph{cluster pattern} $\mathbf{\Sigma}=\{(\mfx_t,B_t)|\ t\in \mbT_n\}$ is a collection of seeds which are labeled by the $n$-regular tree $\mbT_{n}$, such that $(\mfx_{t^{\prime}},B_{t^{\prime}})=\mu_k(\mfx_t,B_t)$ for any $t \stackrel{k}{\longleftrightarrow} t^{\prime}$ in $\mbT_n$. 

In the following, we use the notations that $\mathbf{x}_{t}=(x_{1;t},\dots,x_{n;t})$ and $B_{t}=(b_{ij;t})_{n\times n}$. Note that for an arbitrary fixed vertex $t_0\in \mbT_n$, we call the seed $(\mfx_{t_0},B_{t_0})$ \emph{initial seed} and denote by $\mathbf{x}_{t_0}=\mfx=(x_{1},\dots,x_{n})$ the \emph{initial cluster} and $B_{t_0}=B=(b_{ij})_{n\times n}$ the \emph{initial exchange matrix}.
\begin{definition}[\emph{Cluster algebra}]\label{cluster algebras}
	For a cluster pattern $\mathbf{\Sigma}$, the  \emph{cluster algebra} $\mcA=\mcA(\mathbf{\Sigma})$ is the $\mbQ$-subalgebra of $\mcF$ generated by all the cluster variables $\{x_{i;t}|\ i=1,\dots,n;t\in\mbT_{n}\}$. Here, $n$ is  the \emph{rank} of $\mcA$.\end{definition}
\begin{definition}[\emph{Finite type}]
A cluster algebra $\mcA$ is called of \emph{finite type} if it contains only finite distinct seeds. Otherwise, it is called of \emph{infinite type}.
\end{definition} 
In particular, the cluster algebras of rank 2 as follows are simple but important.
\begin{example}
	Let $\mbT_2$ be a $2$-regular tree as follows \begin{align}
	\dots\stackrel{2}{\longleftrightarrow}t_{-2}\stackrel{1}{\longleftrightarrow}t_{-1}\stackrel{2}{\longleftrightarrow}t_0 \stackrel{1}{\longleftrightarrow}t_1\stackrel{2}{\longleftrightarrow}t_2\stackrel{1}{\longleftrightarrow} \dots  \notag,
\end{align}
	and the initial exchange matrix be  \begin{align}
		B_0=\begin{pmatrix}
	0 & m \\ -n & 0
\end{pmatrix},\ \text{for}\ m,n\in \mbN. 
	\label{2rank}\end{align}  According to \cite[Theorem 1.8]{FZ03}, we get $\mcA$ is of finite type if and only if $mn\leq 3$. 
	\begin{enumerate}
		\item When $m=n=0$, $\mcA=\mcA(0,0)$ is called of \emph{type $A_1\times A_1$}.
		\item When $m=n=1$, $\mcA=\mcA(1,-1)$ is called of \emph{type $A_2$}.
		\item When $m=1,n=2$, $\mcA=\mcA(1,-2)$ is called of \emph{type $B_2$}.
		\item When $m=1,n=3$, $\mcA=\mcA(1,-3)$ is called of \emph{type $G_2$}.
	\end{enumerate}  
The finite type clusters of rank 2 can be referred to \Cref{clusters of finite type}. In addition, the cluster algebras with $mn=4$ are called of \emph{affine type}, denoted by $\mcA(2,-2)$ and $\mcA(1,-4)$ respectively. When $mn\geq 5$, the cluster algebras are called of \emph{non-affine type}.

\end{example}
\begin{definition}[\emph{Finite mutation type}]\label{def finite mutation type}
A cluster algebra $\mcA$ is called
of \emph{finite mutation type} if it contains only finitely many exchange matrices.
\end{definition} Note that the cluster algebra with finitely
many exchange matrices is not necessarily of finite type. In addition, all the coefficient-free cluster algebras of finite mutation type were classified in \cite{FST1, FST2} via unfoldings.
\begin{definition}[\emph{Sign-equivalent}]\label{def sign}
	An exchange matrix $B$ is said to be  \emph{sign-equivalent} if its mutation equivalence class is $[B]=\{B,-B\}$.
\end{definition} It is direct that the cluster algebras with sign-equivalent exchange matrices are of finite mutation type. The matrices \eqref{2rank} without $m=n=0$ are the simplest sign-equivalent exchange matrices. Here, the zero matrix is not sign-equivalent since its mutation equivalence class only contains itself and it is trivial. In addition, the exchange matrices as follows are also sign-equivalent:
\begin{align}
	B=\begin{pmatrix}0 & 2 & -2\\ -2 & 0 & 2\\ 2 & -2 & 0\end{pmatrix},\ \begin{pmatrix}0 & 1 & -1\\ -4 & 0 & 2\\ 4 & -2 & 0\end{pmatrix}.\notag
\end{align} Note that the mutation rules for the sign-equivalent exchange matrices are invariant.
\begin{table}
\centering
\caption{A list of rank 2 clusters  of finite type}
\scalebox{0.7}{
\begin{tabular}{cc|cc|cc|cc}
	\toprule[1.5pt]
	$A_1\times A_1$ & Clusters & $A_2$ & Clusters & $B_2$ & Clusters & $G_2$ & Clusters\\
	\midrule
	1 & $(x_1,x_2)$ & 1 & $(x_1,x_2)$ & 1 & $(x_1,x_2)$ & 1 & $(x_1,x_2)$ \\ 
	2 & $(\frac{2}{x_1},x_2)$ & 2 & $(\frac{x_2+1}{x_1},x_2)$ & 2 & $(\frac{x_2^2+1}{x_1},x_2)$ & 2 & $(\frac{x_2^3+1}{x_1},x_2)$\\
	3 & $(\frac{2}{x_1},\frac{2}{x_2})$ & 3 & $(\frac{x_2+1}{x_1},\frac{x_1+x_2+1}{x_1x_2})$ & 3 & $(\frac{x_2^2+1}{x_1},\frac{x^2_2+x_1+1}{x_1x_2})$ & 3 & $(\frac{x_2^3+1}{x_1},\frac{x_2^3+x_1+1}{x_1x_2})$ \\
	4 & $(x_1,\frac{2}{x_2})$ & 4 & $(\frac{x_1+1}{x_2},\frac{x_1+x_2+1}{x_1x_2})$ & 4 & $(\frac{x^2_2+x_1^2+2x_1+1}{x_1x_2^2},\frac{x^2_2+x_1+1}{x_1x_2})$ & 4 & $(\frac{x_2^6+3x_1x_2^3+2x_2^3+x_1^3+3x_1^2+3x_1+1}{x_1^2x_2^3},\frac{x_2^3+x_1+1}{x_1x_2})$ \\ 
	& & 5 & $(\frac{x_1+1}{x_2},x_1)$ & 5 & $(\frac{x^2_2+x_1^2+2x_1+1}{x_1x_2^2},\frac{x_1+1}{x_2})$ & 5 &   $(\frac{x_2^6+3x_1x_2^3+2x_2^3+x_1^3+3x_1^2+3x_1+1}{x_1^2x_2^3},\frac{x_2^3+x_1^2+2x_1+1}{x_1x_2^2})$\\ 
	& & 6& $(x_2,x_1)$ & 6 & $(x_1,\frac{x_1+1}{x_2})$ & 6 & $(\frac{x_2^3+x_1^3+3x_1^2+3x_1+1}{x_1x_2^3},\frac{x_2^3+x_1^2+2x_1+1}{x_1x_2^2})$\\
	& & 7& $(x_2,\frac{x_2+1}{x_1})$& & & 7 & $(\frac{x_2^3+x_1^3+3x_1^2+3x_1+1}{x_1x_2^3},\frac{x_1+1}{x_2})$\\
	& & 8& $(\frac{x_1+x_2+1}{x_1x_2},\frac{x_2+1}{x_1})$& & &8& $(x_1,\frac{x_1+1}{x_2})$\\
	&& 9 &$(\frac{x_1+x_2+1}{x_1x_2},\frac{x_1+1}{x_2})$ & &\\
	&& 10 & $(x_1,\frac{x_1+1}{x_2})$ & &\\
	\bottomrule[1.5pt]
\end{tabular}
} \label{clusters of finite type}
\end{table}
\subsection{Permutation matrices and irreducible matrices}\

We recall some notions and properties about permutation matrices and irreducible matrices according to \cite[Chapter XIII]{Gan98}. 
\begin{definition}[\emph{Permutation matrix}]
	A permutation matrix $P$ is a square matrix that has exactly one entry of 1 in each row and each column with all other entries 0.
\end{definition}
\begin{definition}[\emph{Reducible matrix}]
Let $B\in \text{Mat}_{n\times n}(\mbR)$.  It is said to be \emph{reducible} if there is a permutation matrix $P$, such that \begin{align}P B P^{T}=\left(\begin{array}{ll}{B_1} & {B_2} \\ {O} & {B_3}\end{array}\right),\notag \end{align} where $B_1\in \text{Mat}_{k\times k}(\mbR)$ and $B_3 \in \text{Mat}_{(n-k)\times (n-k)}(\mbR)$ with $1\leq k\leq n-1$. Otherwise, it is said to be \emph{irreducible}.
\end{definition}
In particular, assume that $B$ is an exchange matrix. Since $B$ is sign-skew-symmetric, it is reducible if and only if there is a permutation matrix $P$, such that \begin{align}P B P^{T}=\left(\begin{array}{ll}{B_1} & {O} \\ {O} & {B_3}\end{array}\right),\notag \end{align} where $B_1\in \text{Mat}_{k\times k}(\mbZ)$ and $B_3 \in \text{Mat}_{(n-k)\times (n-k)}(\mbZ)$. Under this case, it is also called the \emph{direct sum} of the square matrices $B_1$ and $B_3$. Furthermore, for a permutation $\sigma\in \mathfrak{S}_n$, we define the (right) action of $\sigma$ on $B$ by $\sigma(B)=B^{\prime}=(b_{ij}^{\prime})$, where 
\begin{align}
	b^{\prime}_{ij}=b_{\sigma(i)\sigma(j)}.\notag
\end{align} It is easy to check that $\sigma(B)$ is still an exchange matrix. Moreover, there is a permutation matrix $P$ corresponding to $\sigma$, such that 
\begin{align}
	\sigma(B)=PBP^{T}.\notag
\end{align} Indeed, there is an equivalence relation for exchange matrices: $B^{\prime}\sim B$ if there is some permutation $\sigma\in \mathfrak{S}_n$, such that $B^{\prime}=\sigma(B)$. We call $B$ and $B^{\prime}$ \emph{permutation-equivalent}.

\subsection{Mutation invariants of cluster algebras}\

Given a cluster algebra $\mcA$ of rank $n$, we recall the definition of mutation invariants based on \cite[Definition 1.1]{CL24}.
By Laurent phenomenon \cite[Theorem 3.1]{FZ02} and the positivity theorem \cite[Theorem 4.10]{GHKK18}, all the clusters labeled by $t$ can be written as  
$$\mfx_t=(c_{1;t}(x_1,\cdots,x_n),\cdots,c_{n;t}(x_1,\cdots,x_n)),$$ where each $c_{i,t}(x_1,\cdots,x_n)$ is a Laurent polynomial with non-negative integer coefficients.
\begin{definition}[\emph{Mutation invariant}]\label{mutation invariant}
	 A non-constant and reduced rational function $\mcT(x_1,\dots,x_n)\in \mbQ(x_1,\dots,x_n)$ is called a \emph{mutation invariant of $\mcA$} if for any $t \in \mbT_{n}$,
	\begin{equation}
		\mcT(x_{1},\cdots,x_{n})=\mcT(x_{1;t},\cdots,x_{n;t}).\notag
	\end{equation}
	In particular, if $\mcT(x_1,\dots,x_n) \in \mbQ[x_{1}^{\pm 1},\dots,x_{n}^{\pm 1}]$, it is called a \emph{Laurent mutation invariant of $\mcA$}. 
\end{definition}

In fact, if $\mcT(x_1,\dots,x_n)$ is a mutation invariant, then so does $F(\mcT(x_1,\dots,x_n))$, where $F(X)\in \mbQ[X]$ is a non-constant polynomial. For example, according to \cite[Lemma 2.23 \& Lemma 2.26]{CL24}, we have the Laurent mutation invariants of $\mcA(2,-2)$ and $\mcA(1,-4)$ respectively:
\begin{align}
	F\left(\dfrac{x_1^2+x_2^2+1}{x_1x_2}\right),\  F\left(\dfrac{x_2^4+x_1^2+2x_1+1}{x_1x_2^2}\right).\notag
\end{align}
\begin{remark}\label{IMR rules}
	If $B$ is sign-equivalent, then by \cite[Lemma1.5]{CL24}, $\mcT(x_1,\dots,x_n)$ is a mutation invariant if and only if $\mcT(x_1,\dots,x_n)=\mcT(\mu_k(x_1,\dots,x_n))$ for any $k\in \{1,\dots,n\}$. Hence, it is important to classify all the irreducible sign-equivalent exchange matrices.
\end{remark}

\section{The classification of sign-equivalent exchange matrices}\label{S3}
All the coefficient-free cluster algebras of finite mutation type were classified in \cite{FST1, FST2}. In particular, the sign-equivalent exchange matrices are mutation-finite. In this section, we calculate and  classify all the irreducible sign-equivalent exchange matrices by cluster mutations since they are important to finding more mutation invariants of cluster algebras. 

First of all, we exhibit the compatibility between permutations and cluster mutations as follows.
\begin{proposition}\label{permutation}
	Let $\sigma\in \mathfrak{S}_n$ and $k\in \{1,\dots,n\}$. Then, the following equality holds:
	\begin{align}
		\mu_{\sigma^{-1}(k)}(\sigma(B))=\sigma(\mu_k(B)).\notag
	\end{align}
\end{proposition}
\begin{proof}
	Let $B=(b_{ij})_{n\times n}$, $B^{\prime}=\mu_k(B)=(b^{\prime}_{ij})_{n\times n}$ and $\sigma(\mu_k(B))=(c_{ij})_{n\times n}$. Then, on the right hand side, by the mutation rules \eqref{matrix mutation} and the definition of permutation $\sigma$, we have $c_{ij}=b^{\prime}_{\sigma(i)\sigma(j)}$, where 
	\begin{align}
		b^{\prime}_{\sigma(i)\sigma(j)}=\left\{
		\begin{array}{ll}
			-b_{\sigma(i)\sigma(j)},\  \text{if}\   i=\sigma^{-1}(k) \;\;\mbox{or}\;\; j=\sigma^{-1}(k), \\
			b_{\sigma(i)\sigma(j)}+[b_{\sigma(i)\sigma(k)}]_{+}b_{\sigma(k)\sigma(j)}+b_{\sigma(i)\sigma(k)}[-b_{\sigma(k)\sigma(j)}]_{+},\ \text{if} \ i\neq \sigma^{-1}(k) \; \\
			\mbox{and}\; j\neq \sigma^{-1}(k). 
		\end{array} \right. \notag
	\end{align} On the left hand side, we assume that $\sigma(B)=(d_{ij})_{n\times n}$ and $\mu_{\sigma^{-1}(k)}(\sigma(B))=(d^{\prime}_{ij})_{n\times n}$. Then, it implies that $d_{ij}=b_{\sigma(i)\sigma(j)}$. Moreover, we have 
	\begin{align}
		d_{ij}^{\prime}=\left\{
		\begin{array}{ll}
			-d_{ij}, &  \text{if}\ i=\sigma^{-1}(k) \;\;\mbox{or}\;\; j=\sigma^{-1}(k), \\
			d_{ij}+[d_{ik}]_{+}d_{kj}+d_{ik}[-d_{kj}]_{+}, &  \text{if}\ i\neq \sigma^{-1}(k) \;
			\mbox{and}\; j\neq \sigma^{-1}(k). 
		\end{array} \right. \notag
	\end{align}  Hence, we get 
	\begin{align}
		d^{\prime}_{ij}=\left\{
		\begin{array}{ll}
			-b_{\sigma(i)\sigma(j)},\  \text{if}\   i=\sigma^{-1}(k) \;\;\mbox{or}\;\; j=\sigma^{-1}(k), \\
			b_{\sigma(i)\sigma(j)}+[b_{\sigma(i)\sigma(k)}]_{+}b_{\sigma(k)\sigma(j)}+b_{\sigma(i)\sigma(k)}[-b_{\sigma(k)\sigma(j)}]_{+},\ \text{if} \ i\neq \sigma^{-1}(k) \;\\
			\mbox{and}\; j\neq \sigma^{-1}(k), 
		\end{array} \right. \notag
	\end{align} which implies that $d^{\prime}_{ij}=c_{ij}$ for any $i,j\in \{1,\dots,n\}$ and $\mu_{\sigma(k)}(\sigma(B))=\sigma(\mu_k(B))$.
\end{proof}
Then, according to \Cref{permutation}, we obtain the proposition as follows directly.
\begin{corollary}\label{sigma}
	Let $\sigma\in \mathfrak{S}_n$. If $[B]=\{B,-B\}$, then $[\sigma(B)]=\{\sigma(B),-\sigma(B)\}$.
\end{corollary}
\begin{remark}
	In \cite[Definition 2.8 \& Proposition 2.9]{Nak23}, the (left) action of $\sigma\in \mathfrak{S}_n$ on $B$ is defined by $\sigma(B)=B^{\prime}=(b^{\prime}_{ij})_{n\times n}$, where $b^{\prime}_{ij}=b_{\sigma^{-1}(i)\sigma^{-1}(j)}$. Then, the equality holds: $\mu_{\sigma(k)}(\sigma(B))=\sigma(\mu_k(B))$.
\end{remark}
\begin{lemma}\label{lem3}
	Let $B$ be an irreducible sign-equivalent exchange matrix. Then, $\mu_i(B)=-B$ for any $i\in \{1,\dots,n\}$.
\end{lemma}
\begin{proof}
	Note that $B$ is sign-equivalent. If $\mu_i(B)=B$ for some $i\in \{1,\dots,n\}$, by the mutation rules \eqref{matrix mutation}, we have $b_{il}=-b_{il}=0$ for any $l\in \{1,\dots,n\}$. Hence, we conclude that all the entries of $i$-th row and $i$-th column of $B$ are zero, which implies that it is reducible. Then, it contradicts with $B$ is irreducible.
\end{proof}
Now, we are ready to classify all the irreducible sign-equivalent exchange matrices.
\begin{theorem}\label{main}
	There are four types of irreducible sign-equivalent exchange matrices are as follows:
	\begin{enumerate}
		\item $B=\sigma(B_0)$, where $\sigma\in \mathfrak{S}_2$ and $B_0=\begin{pmatrix}
	0 & b \\ -c & 0 \end{pmatrix}$ for $b,c\in \mbN_+$.
		\item $B=\sigma(B_0)$, where $\sigma\in \mathfrak{S}_3$ and $B_0=\begin{pmatrix}0 & 2 & -2\\ -2 & 0 & 2\\ 2 & -2 & 0\end{pmatrix}$.
		\item $B=\sigma(B_0)$, where $\sigma\in \mathfrak{S}_3$ and $B_0=\begin{pmatrix}0 & 1 & -1\\ -4 & 0 & 2\\ 4 & -2 & 0\end{pmatrix}$.
		\item $B=\sigma(B_0)$, where $\sigma\in \mathfrak{S}_3$ and $B_0=\begin{pmatrix}0 & -4 & 4\\ 1 & 0 & -2\\ -1 & 2 & 0\end{pmatrix}$.
	\end{enumerate}
\end{theorem}
\begin{proof}
	Firstly, we prove that there is no irreducible sign-equivalent exchange matrix for $n\geq 4$. Assume that $B=(b_{ij})_{n\times n}$. By \Cref{lem3}, we have $\mu_i(B)=-B$ for $i\in \{1,\dots,n\}$. 
	\begin{itemize}[leftmargin=1em]\itemsep=0pt
	\item If there are $i\neq j$ such that $b_{ij}= 0$, then we take the $i$-direction mutation of $B$ and let $\mu_i(B)=(b_{ij}^{\prime})_{n\times n}$. Note that for any $l\in \{1,\dots,n\}$, we have 
	\begin{align}
		b_{lj}^{\prime}=b_{lj}+\sgn(b_{li})[b_{li}b_{ij}]_+=b_{lj}.\notag
	\end{align} Since $\mu_i(B)=-B$, we get $b_{lj}={b_{lj}}^{\prime}=0$. Then, we have $b_{jl}=0$, which implies that all the entries of $j$-th row and $j$-th column of $B$ are zero. Hence, $B$ is reducible.
	\item If $b_{ij}\neq 0$ for all $i\neq j$, since $n\geq 4$, there are $i_1, i_2\in \{2,\dots,n\}$, such that $i_1\neq i_2$ and 
	\begin{align}
		\sgn(b_{1i_1})=-\sgn(b_{i_{2}1}),\  \sgn(b_{i_{1}1})=-\sgn(b_{1i_{2}}).\notag
	\end{align} We take the $1$-direction mutation of $B$ and let $\mu_1(B)=(b_{ij}^{\prime})_{n\times n}$. Then, we have 
	\begin{align}
		b_{i_1i_2}^{\prime}=b_{i_1i_2}+\sgn(b_{i_{1}1})[b_{i_{1}1}b_{1i_2}]_+=b_{i_1i_2}.\notag
	\end{align} Since $\mu_1(B)=-B$, we get $b_{i_1i_2}=b_{i_1i_2}^{\prime}=0$. Then, it contradicts with $b_{ij}\neq 0$ for all $i\neq j$.
		
	Secondly, we classify all the irreducible sign-equivalent exchange matrices for $n\leq 3$. It is direct that all the exchange matrices of order 2 without $b=c=0$ (type $A_1\times A_1$) are irreducible and sign-equivalent. Hence, we only need to focus on the case that $n=3$. Assume that $B$ is an irreducible sign-equivalent exchange matrix as follows:
\begin{align}
	B=\begin{pmatrix}0 & b_{12} & b_{13}\\ b_{21} & 0 & b_{23}\\ b_{31} & b_{32} & 0\end{pmatrix}=\begin{pmatrix}0 & a & b\\ c & 0 & d\\ e & f & 0\end{pmatrix},\notag
\end{align} where $ac\leq 0, be\leq 0$ and $df\leq 0$. By \Cref{lem3}, we have $\mu_i(B)=-B$ for $i=1,2,3$, which implies that 
\begin{align}
	\left\{
		\begin{array}{ll}
			4d+c|b|+|c|b=0,\\
			 4f+e|a|+|e|a=0,\\
			 4b+a|d|+|a|d=0,\\
			 4e+f|c|+|f|c=0,\\
			 4a+b|f|+|b|f=0,\\
			 4c+d|e|+|d|e=0.
		\end{array} \right.\label{*}
\end{align} Note that $B$ is sign-equivalent. There are only three cases to be discussed as follows:
\begin{align}
	(1)\ a\geq 0, b\geq 0, d\geq 0;\ (2)\ a\geq 0, b\geq 0, d\leq 0;\ (3)\ a\geq 0, b\leq 0, d\geq 0.\notag
\end{align}

Case (1): If $a\geq 0, b\geq 0, d\geq 0$, then the equation \eqref{*} 
implies that $a=b=c=d=e=f=0$, that is $B=O$. Hence, it is reducible and there is a contradiction.

Case (2): If $a\geq 0, b\geq 0, d\leq 0$, then the equation \eqref{*} implies that $d=f=b=e=a=c=0$. 
Hence, we get $B=O$ and it is reducible. Note that the case $a\leq 0,\ b\geq 0,\  d\geq 0$ is similar and it also implies that $B=O$.

Case (3): If $a\geq 0, b\leq 0, d\geq 0$, the equation \eqref{*} can be reduced to 
\begin{align}
	\left\{
		\begin{array}{ll}
			2d=bc,\\
			 2f=-ac,\\
			 2b=-ad,\\
			 2e=fc,\\
			 2a=bf,\\
			 2c=-de.
		\end{array}\right.\label{**}
\end{align} Multiplying both sides of the  \Cref{**} together, we have 
\begin{align}
	64abcdef=-(abcdef)^2.\notag
\end{align} It implies that $abcdef=0$ or $abcdef=-64$. If $abcdef=0$, without loss of generality, we might assume that $a=c=0$. Hence, by \eqref{**}, we get $B=O$, which contradicts with that $B$ is irreducible. Therefore, we have $abcdef=-64$. Now, we claim that $a\leq 4$. Otherwise, assume that $a>4$ and then $bf=2a>8$. Then, we get 
\begin{align}
	abcdef=-2a(bf)c^2<-64\notag
\end{align} and it contradicts with $abcdef=-64$. Since $a$ is a factor of $64$, we have $a\in \{4,2,1\}$.
\begin{enumerate}
	\item[(i)] If $a=4$, by \eqref{**}, we get $bf=4de=8$, which implies that $c=-1$. Hence, the equation \eqref{**} can be reduced to $2d=-b,\ 2e=-f,\ de=2.$ Then, it implies that $(d,e,b,f)=(2,1,-4,-2)$ or $(1,2,-2,-4)$.
	Moreover, we have
	\begin{align}
		B=\begin{pmatrix}0 & 4 & -4\\ -1 & 0 & 2\\ 1 & -2 & 0\end{pmatrix}\ \text{or}\ \begin{pmatrix}0 & 4 & -2\\ -1 & 0 & 1\\ 2 & -4 & 0\end{pmatrix}.\label{case1}
	\end{align}
	\item[(ii)] If $a=2$, by \eqref{**}, we get $bf=de=4$, which implies that $c=-2$. Hence, the equation \eqref{**} can be reduced to $d=-b,\ e=-f,\ de=4.$ 
Then, it implies that $(d,e,b,f)=(2,2,-2,-2)$, $(1,4,-1,-4)$ or $(4,1,-4,-1)$.
	Moreover, we have
	\begin{align}
		B=\begin{pmatrix}0 & 4 & -4\\ -1 & 0 & 2\\ 1 & -2 & 0\end{pmatrix},\ \begin{pmatrix}0 & 4 & -2\\ -1 & 0 & 1\\ 2 & -4 & 0\end{pmatrix} \ \text{or}\ \begin{pmatrix}0 & 4 & -2\\ -1 & 0 & 1\\ 2 & -4 & 0\end{pmatrix}. 	\label{case2}\end{align} 
	\item[(iii)] If $a=1$, by \eqref{**}, we get $4bf=de=8$, which implies that $c=-4$. Hence, the equation \eqref{**} can be reduced to $d=-2b,\ e=-2f,\ bf=2$. 
Then, it implies that $(d,e,b,f)=(4,2,-2,-1)$ or $(2,4,-1,-2)$.
	Moreover, we have
	\begin{align}
		B=\begin{pmatrix}0 & 1 & -2\\ -4 & 0 & 4\\ 2 & -1 & 0\end{pmatrix}\ \text{or}\ \begin{pmatrix}0 & 1 & -1\\ -4 & 0 & 2\\ 4 & -2 & 0\end{pmatrix}.\label{case3}
	\end{align}
\end{enumerate} In conclusion, it is direct to check that all the matrices in \eqref{case1}, \eqref{case2} and \eqref{case3} are skew-symmetrizable and sign-equivalent. Moreover, all of them can be expressed by one of the following: 
\begin{align}
	\sigma\begin{pmatrix}0 & 2 & -2\\ -2 & 0 & 2\\ 2 & -2 & 0\end{pmatrix},\ \sigma\begin{pmatrix}0 & 1 & -1\\ -4 & 0 & 2\\ 4 & -2 & 0\end{pmatrix},\ \sigma\begin{pmatrix}0 & -4 & 4\\ 1 & 0 & -2\\ -1 & 2 & 0\end{pmatrix},\notag
\end{align} where $\sigma\in \mathfrak{S}_3$.
\end{itemize}
\end{proof}
\begin{definition}\label{ise}
	We denote by $X_{\ise}$ the set of all irreducible sign-equivalent exchange matrices in \Cref{main}.
\end{definition}
\begin{corollary}
	All the sign-equivalent exchange matrices $B$ are permutation-equivalent to the direct sum of only one of $X_{\ise}$ and (possibly none) zero matrix $O$.  
\end{corollary}
\begin{proof}
	 Let $B\neq O_{n\times n}$. Then, there exists a permutation matrix $P$ corresponding to the permutation $\sigma\in \mathfrak{S}_n$, such that 
	\begin{align}
		\sigma(B)=PBP^{T}=
	\text{diag}(B_1,B_2,\dots,B_r,O)\notag
	\end{align} is a block diagonal matrix, where $r\geq 1$ and $B_i$ is irreducible for $i=1,\dots,r$. Note that any mutation $\mu_i$ essentially only mutates one of the blocks of $\sigma(B)$ and keeps others invariant. Hence, by \Cref{sigma} and \Cref{lem3}, we conclude that $r=1$ and $B_1$ is irreducible sign-equivalent. Then, according to \Cref{main}, we get $B_1\in X_{\ise}$.
\end{proof}
\begin{remark}
	In fact, according to the proof of \Cref{main}, the set of all the sign-equivalent sign-skew-symmetric matrices is also $X_{\ise}$ because we only use the symmetry of signs. 
\end{remark}

\section{A Diophantine explanation for rank 2 cluster algebras of finite and affine type}\label{S4}
In this section, we aim to give an explanation for the differences between finite type and affine type cluster algebras of rank 2 by diophantine equations. Firstly, we recall an important theorem given by Chen-Li in \cite{CL24}. Assume that $\mcA$ is of finite type with $m$ clusters 
	$(c_{1;i}(x_1,x_2),c_{2;i}(x_1,x_2)),i=1,\dots,m.$ Without loss of generality, we label them same as  \Cref{clusters of finite type}.
\begin{theorem}[{\cite[Theorem 3.1]{CL24}}]\label{mutation invariant of rank 2}
	A non-constant rational function $\mcT(x_1,x_2)$ is a mutation invariant of $\mcA$ if and only if there is a symmetric polynomial $\Phi(u_1,\cdots,u_m)$ over $\mbQ$ and a rational function $F(X_1,X_2)$, such that 
	\begin{align}\label{main formula}
		\mcT(x_1,x_2)=&{}\Phi(F(c_{1;1}(x_1,x_2),c_{2;1}(x_1,x_2)),\cdots,F(c_{1;m}(x_1,x_2),c_{2;m}(x_1,x_2))).
	\end{align}
\end{theorem}
In particular, when $\Phi(u_1,\cdots,u_m)\in \mbQ_{\geq 0}[u_1,\cdots,u_m]$ and $F(X_1,X_2)\in \mbQ_{\geq 0}[X_1,X_2]$ are both non-constant, it is direct that $\mcT(x_1,x_2)$ is a non-constant Laurent polynomial over $\mbQ_{\geq 0}$. Then a class of Diophantine equations with the initial solution $(a,b)$ are \begin{align}
	\mcT(x_1,x_2)=\mcT(a,b),\label{finite equation}
\end{align} where $(a,b)\in \mbN_+^2$. Note that these equations can be adjusted to with integer coefficients. In \cite{CL24}, there are several important Diophantine equations with finite type cluster algebraic structures of rank 2 as follows:
\begin{enumerate}
    \item (Type $A_1\times A_1$) $x_1^2x_2+x_1x_2^2+2x_1+2x_2=6x_1x_2.$
	\item (Type $A_2$) $x_{1}^{2}x_{2}+x_{1}x_{2}^{2}+x_{1}^2+x_{2}^2+2x_{1}+2x_{2}+1=9x_{1}x_{2}.$
	\item (Type $B_2$) $x_{2}^4+x_{1}^2x_{2}^2+2x_{2}^2+x_{1}^2+2x_{1}+1=8x_{1}x_{2}^2.$
	\item (Type $G_2$) $x_{2}^4+x_{1}x_{2}^3+x_{2}^3+x_{1}^2x_{2}+2x_{1}x_{2}+x_{1}^2+x_{2}+2x_{1}+1=11x_{1}x_{2}^2.$
\end{enumerate} 
According to \cite[Lemma 3.1, Lemma 3.3, Lemma 3.5 \& Lemma 3.6]{CL24}, there are only finite positive integer solutions to the four equations above. Furthermore, all the solutions to them can be generated by the initial solution $(1,1)$ through the corresponding cluster mutations. 
\begin{definition}\label{reductive}
	Let $L(x_1,x_2)$ be a (reduced) Laurent polynomial over $\mbQ_{\geq 0}$ as follows \begin{align}
		L(x_1,x_2)=\dfrac{f(x_1,x_2)}{x_1^ux_2^v},\notag
	\end{align} where $f(x_1,x_2)\in \mbQ_{\geq 0}[x_1,x_2]$ and $u,v\in \mbN$. We call $L(x_1,x_2)$ \emph{reductive} if there is a term $x_1^{u^{\prime}}x_2^{v^{\prime}}$ in $f(x_1,x_2)$, such that $u^{\prime}\geq u$, $v^{\prime}\geq v$ and $(u^{\prime},v^{\prime})\neq (u,v)$. More precisely, there is a sum term of polynomial $x_1^{u-u^{\prime}}x_2^{v-v^{\prime}}$,  such that $L(x_1,x_2)=x_1^{u-u^{\prime}}x_2^{v-v^{\prime}}+L^{\prime}(x_1,x_2)$, where $L^{\prime}(x_1,x_2)\in \mbQ_{\geq 0}[x_{1}^{\pm 1},x_{2}^{\pm 1}]$.
\end{definition}
\begin{example}\label{reductive ex}
	The corresponding four Laurent polynomials of type $A_1\times A_1, A_2, B_2$ and $G_2$ above are  reductive. In particular, all the non-constant polynomials in $\mbQ_{\geq 0}[x_1,x_2]$ are reductive.
\end{example}
\begin{lemma}\label{closed}
	Let $L_1(x_1,x_2), L_2(x_1,x_2) \in \mbQ_{\geq 0}[x_{1}^{\pm 1},x_{2}^{\pm 1}]$. If $L_1(x_1,x_2)$ is reductive, then $L_1(x_1,x_2)+L_2(x_1,x_2)$ is also reductive. Moreover, if $L_2(x_1,x_2)$ is also reductive, then $L_1(x_1,x_2)L_2(x_1,x_2)$ is also reductive. 
\end{lemma}
\begin{proof}
	Note that both $L_1(x_1,x_2)$ and $L_2(x_1,x_2)$ are subtraction-free. It can be obtained by direct calculation and the definition of reductive Laurent polynomials.
\end{proof}
In the following, we aim to give a main theorem about the finiteness of the \Cref{finite equation}. 
\begin{theorem}\label{finite solution}
	There are only finite positive integer solutions to the \Cref{finite equation}.
\end{theorem}
\begin{proof} We will take two steps to complete the proof.\\
$\mathbf{Step 1}:$ We claim that if $L(x_1,x_2)$ is reductive, then there are finite positive integer solutions to the equation \begin{align}
	L(x_1,x_2)=L(a,b),\label{L equation}
\end{align} where $(a,b)\in \mbN_{+}^2$. In fact, the equation is equivalent to \begin{align}
	f(x_1,x_2)=L(a,b)x_1^ux_2^v.\notag
\end{align} Since $L(x_1,x_2)$ is reductive, we might assume that there is a term $x_1^{u^{\prime}}x_2^{v^{\prime}}$ in $f(x_1,x_2)$, such that $u^{\prime}> u$ and $v^{\prime}\geq v$. Hence, there must be an upper bound $M_0\in \mbN_+$ for $x_1$, that is $x_1\leq M_0$. This infers that there are only finite possible values of $x_1$ to the \Cref{L equation}. Moreover, the choices of $x_2$ are also finite after fixing $x_1$ and the claim holds.
\\
$\mathbf{Step 2}:$ Now, we focus on proving that $\mcT(x_1,x_2)$ in the \Cref{finite equation} is reductive. In what follows, without loss of generality, we can regard the non-zero coefficients as $1$ since they are all positive and will not influence the results.\\
$(1)$ If the symmetric polynomial $\Phi(u_1,\cdots,u_m)$ has a single term $u_i^\alpha$, where $\alpha\geq 1$. Then, it must have a term $u_1^\alpha$. Note that $F(X_1,X_2)\in \mbQ_{\geq 0}[X_1,X_2]$ and \begin{align}
		(c_{1;1}(x_1,x_2),c_{2;1}(x_1,x_2))=(x_1,x_2).\notag
	\end{align} Hence, according to  formula \eqref{main formula}, $\mcT(x_1,x_2)$ contains a term $(F(x_1,x_2))^{\alpha}\in \mbQ_{\geq 0}[x_1,x_2]$, which is reductive. By \Cref{closed}, we get $\mcT(x_1,x_2)$ is also reductive.\\
$(2)$ If the symmetric polynomial $\Phi(u_1,\cdots,u_m)$ has no single term $u_i^\alpha$, where $\alpha\geq 1$. \\ $(2.1)$ When $m=10,6,8$, we claim that there exists a \emph{cluster sequence} $\mcL_m=(l_1,l_2,\dots,l_m)$, where $l_i\in \{1,\dots,m\}$ and $l_i\neq l_j(i\neq j)$, such that the monomials \begin{align}u_{l_1}u_{l_2},\ u_{l_1}u_{l_2}u_{l_3},\ \dots,\ u_{l_1}u_{l_2}u_{l_3}\dots u_{l_m}\notag\end{align} are reductive after substituting each $u_{l_j}$ by $F(c_{1;l_j}(x_1,x_2),c_{2;l_j}(x_1,x_2))$ simultaneously (See \Cref{clusters of finite type}). Here, we denote them by $u_{l_1}u_{l_2}|_{F},\dots, u_{l_1}u_{l_2}u_{l_3}\dots u_{l_m}|_{F}$. In fact, we can take \begin{align}F_1(X_1,X_2)=X_1,\ F_2(X_1,X_2)=X_2.\notag\end{align} 
In what follows, we always take $l_1=1$ and denote $F_i(c_{1;l_j}(x_1,x_2),c_{2;l_j}(x_1,x_2))$ by $F_{i,j}\ (\text{for}\ i=1,2)$.
\begin{itemize}[leftmargin=1em]\itemsep=0pt
\item Type $A_2\, (m=10)$: Take the cluster sequence $\mcL_{10}=(1,2,6,7,5,10,3,8,4,9)$. Then, we get $$\begin{array}{l}
	F_{1,1}F_{1,2}=x_2+1,\
	F_{2,1}F_{2,2}=x_2^2,\\
	F_{1,1}F_{1,2}F_{1,6}=x_2(x_2+1),\
	F_{2,1}F_{2,2}F_{2,6}=x_1x_2^2,\\
	F_{1,1}F_{1,2}F_{1,6}F_{1,7}=x_2^2(x_2+1),\
	F_{2,1}F_{2,2}F_{2,6}F_{2,7}=x_2^2(x_2+1),\\
	F_{1,1}F_{1,2}F_{1,6}F_{1,7}F_{1,5}=x_2(x_2+1)(x_1+1),\\
	F_{2,1}F_{2,2}F_{2,6}F_{2,7}F_{2,5}=x_1x_2^2(x_2+1),\\
	F_{1,1}F_{1,2}F_{1,6}F_{1,7}F_{1,5}F_{1,10}=x_1x_2(x_2+1)(x_1+1),\\
	F_{2,1}F_{2,2}F_{2,6}F_{2,7}F_{2,5}F_{2,10}=x_1x_2(x_2+1)(x_1+1),\\
	F_{1,1}F_{1,2}F_{1,6}F_{1,7}F_{1,5}F_{1,10}F_{1,3}=x_2(x_2+1)^2(x_1+1),\\
	F_{2,1}F_{2,2}F_{2,6}F_{2,7}F_{2,5}F_{2,10}F_{2,3}=(x_2+1)(x_1+1)(x_1+x_2+1),\\
	F_{1,1}F_{1,2}F_{1,6}F_{1,7}F_{1,5}F_{1,10}F_{1,3}F_{1,8}=(x_2+1)(x_1+1)(x_1+x_2+1),\\
	F_{2,1}F_{2,2}F_{2,6}F_{2,7}F_{2,5}F_{2,10}F_{2,3}F_{2,8}=\frac{(x_2+1)^2(x_1+1)(x_1+x_2+1)}{x_1},\\
	F_{1,1}F_{1,2}F_{1,6}F_{1,7}F_{1,5}F_{1,10}F_{1,3}F_{1,8}F_{1,4}=\frac{(x_2+1)(x_1+1)^2(x_1+x_2+1)}{x_2},\\
	F_{2,1}F_{2,2}F_{2,6}F_{2,7}F_{2,5}F_{2,10}F_{2,3}F_{2,8}F_{2,4}=\frac{(x_2+1)^2(x_1+1)(x_1+x_2+1)^2}{x_1^2x_2},\\
	F_{1,1}F_{1,2}F_{1,6}F_{1,7}F_{1,5}F_{1,10}F_{1,3}F_{1,8}F_{1,4}F_{1,9}=\frac{(x_2+1)(x_1+1)^2(x_1+x_2+1)^2}{x_1x_2^2},\\
	F_{2,1}F_{2,2}F_{2,6}F_{2,7}F_{2,5}F_{2,10}F_{2,3}F_{2,8}F_{2,4}F_{2,9}=\frac{(x_2+1)^2(x_1+1)^2(x_1+x_2+1)^2}{x_1^2x_2^2}.
\end{array}$$

\item Type $B_2\, (m=6)$: Take the cluster sequence $\mcL_6=(1,2,6,7,5,4)$. Then, we get 
$$\begin{array}{l}
	F_{1,1}F_{1,2}=x_2^2+1, \ F_{2,1}F_{2,2}=x_2^2,
	\\F_{1,1}F_{1,2}F_{1,6}=x_1(x_2^2+1),\ F_{2,1}F_{2,2}F_{2,6}=x_2(x_1+1),\\
F_{1,1}F_{1,2}F_{1,6}F_{1,3}=(x_2^2+1)^2,\ 
	F_{2,1}F_{2,2}F_{2,6}F_{2,3}=\frac{(x_2^2+x_1+1)(x_1+1)}{x_1},\\
	F_{1,1}F_{1,2}F_{1,6}F_{1,3}F_{1,5}=\frac{(x_2^2+1)^2(x_2^2+x_1^2+2x_1+1)}{x_1x_2^2},\\
	F_{2,1}F_{2,2}F_{2,6}F_{2,3}F_{2,5}=\frac{(x_1+1)^2(x_2^2+x_1+1)}{x_1x_2},\\
	F_{1,1}F_{1,2}F_{1,6}F_{1,3}F_{1,5}F_{1,4}=\frac{(x_2^2+1)^2(x_2^2+x_1^2+2x_1+1)^2}{x_1^2x_2^4},\\
	F_{2,1}F_{2,2}F_{2,6}F_{2,3}F_{2,5}F_{2,4}=\frac{(x_1+1)^2(x_2^2+x_1+1)^2}{x_1^2x_2^2}.
\end{array}$$
\item Type $G_2\, (m=8)$: Take the cluster sequence $\mcL_8=(1,2,8,3,7,6,5,4)$. Then, we get 
$$\begin{array}{l}
	F_{1,1}F_{1,2}=x_2^3+1,\
	F_{2,1}F_{2,2}=x_2^2,\\
	F_{1,1}F_{1,2}F_{1,8}=x_1(x_2^3+1),\
	F_{2,1}F_{2,2}F_{2,8}=x_2(x_1+1),\\
	F_{1,1}F_{1,2}F_{1,8}F_{1,3}=(x_2^3+1)^2,\
	F_{2,1}F_{2,2}F_{2,8}F_{2,3}=\frac{(x_1+1)(x_2^3+x_1+1)}{x_1},\\
	F_{1,1}F_{1,2}F_{1,8}F_{1,3}F_{1,7}=\frac{(x_2^3+1)^2(x_2^3+x_1^3+3x_1^2+3x_1+1)}{x_1x_2^3},\\
	F_{2,1}F_{2,2}F_{2,3}F_{2,8}F_{2,7}=\frac{(x_1+1)^2(x_2^3+x_1+1)}{x_1x_2},\\
	F_{1,1}F_{1,2}F_{1,3}F_{1,8}F_{1,7}F_{1,6}=\frac{(x_2^3+1)^2(x_2^3+x_1^3+3x_1^2+3x_1+1)^2}{x_1^2x_2^6},\\
	F_{2,1}F_{2,2}F_{2,3}F_{2,8}F_{2,7}F_{2,6}=\frac{(x_1+1)^2(x_2^3+x_1+1)(x_2^3+x_1^2+2x_1+1)}{x_1^2x_2^3},\\
	F_{1,1}F_{1,2}F_{1,3}F_{1,8}F_{1,7}F_{1,6}F_{1,5}=\frac{(x_2^3+1)^2(x_2^3+x_1^3+3x_1^2+3x_1+1)^2(x_2^6+3x_1x_2^3+2x_2^3+x_1^3+3x_1^2+3x_1+1)}{x_1^4x_2^9},\\
	F_{2,1}F_{2,2}F_{2,3}F_{2,8}F_{2,7}F_{2,6}F_{2,5}=\frac{(x_1+1)^2(x_2^3+x_1+1)(x_2^3+x_1^2+2x_1+1)^2}{x_1^3x_2^5},\\
	F_{1,1}F_{1,2}F_{1,3}F_{1,8}F_{1,7}F_{1,6}F_{1,5}F_{1,4}=\frac{(x_2^3+1)^2(x_2^3+x_1^3+3x_1^2+3x_1+1)^2(x_2^6+3x_1x_2^3+2x_2^3+x_1^3+3x_1^2+3x_1+1)^2}{x_1^6x_2^{12}},\\
	F_{2,1}F_{2,2}F_{2,3}F_{2,8}F_{2,7}F_{2,6}F_{2,5}F_{2,4}=\frac{(x_1+1)^2(x_2^3+x_1+1)^2(x_2^3+x_1^2+2x_1+1)^2}{x_1^4x_2^6}.\\
\end{array}$$
\end{itemize} According to the three calculating results for type $A_2, B_2$ and $G_2$ as above, we conclude that $\prod_{j=1}^{k} F_{i,l_j}$ are always reductive, where $i\in \{1,2\}$ and $k\in \{2,\dots,m\}$. Since $F(X_1,X_2)$ is a non-constant polynomial over $\mbQ_{\geq 0}$, according to \Cref{closed}, the claim holds. 

Note that $\Phi(u_1,\dots,u_m)$ is a non-constant symmetric polynomial over $\mbQ_{\geq 0}$. There must be a term corresponding to the sequence $\mcL_m=(l_1,l_2,\dots,l_m)$ in the form of 
\begin{align}
	u_{l_1}^{\alpha_1}u_{l_2}^{\alpha_2}\dots u_{l_m}^{\alpha_m},\ \alpha_1\geq \alpha_2 \geq \dots \geq \alpha_m.
\notag\end{align} It can be written as the products of $u_{l_1},u_{l_1}u_{l_2}, u_{l_1}u_{l_2}u_{l_3}, \dots, u_{l_1}u_{l_2}u_{l_3}\dots u_{l_m}$, where $u_{l_1}=u_1$. Therefore, by \Cref{closed} and the claim above, we obtain a reductive term of $\mcT(x_1,x_2)$: \begin{align}u_{l_1}^{\alpha_1}u_{l_2}^{\alpha_2}\dots u_{l_m}^{\alpha_m}|_{F}=(F(c_{1;l_1}(x_1,x_2),c_{2;l_1}(x_1,x_2)))^{\alpha_1}\dots (F(c_{1;l_m}(x_1,x_2),c_{2;l_m}(x_1,x_2)))^{\alpha_m},\notag\end{align} which implies that $\mcT(x_1,x_2)$ is also reductive. 
\\
$(2.2)$ When $m=4$, that is of type $A_1\times A_1\, $: Take the sequence $\mcL_4=(1,2,4,3)$. Then, we get 
$$\begin{array}{l}
	F_{1,1}F_{1,2}=2, \ F_{2,1}F_{2,2}=x_2^2,
	\\F_{1,1}F_{1,2}F_{1,4}=2x_1,\ F_{2,1}F_{2,2}F_{2,4}=2x_2,\\
	F_{1,1}F_{1,2}F_{1,4}F_{1,3}=4,\ 
	F_{2,1}F_{2,2}F_{2,4}F_{2,3}=4.
\end{array}$$ It implies that the constants $F_{1,1}F_{1,2}$, $F_{1,1}F_{1,2}F_{1,4}F_{1,3}$ and $F_{2,1}F_{2,2}F_{2,4}F_{2,3}$ are not reductive. Since $\Phi(u_1,u_2,u_3,u_4)$ is a non-constant symmetric polynomial, there must be a sum term as follows:
\begin{align}
	u_1^{\alpha_1}u_2^{\alpha_2}u_4^{\alpha_4}u_3^{\alpha_3},\ \alpha_1\geq \alpha_2\geq \alpha_4\geq \alpha_3.\notag
\end{align} Note that it can be written as the products of $u_1,u_1u_2, u_1u_2u_4,  u_1u_2u_4u_3$. If the products contain $u_1$ or $u_1u_2u_4$, then by the calculation above and \Cref{closed}, we have both $u_1^{\alpha_1}u_2^{\alpha_2}u_4^{\alpha_4}u_3^{\alpha_3}|_{F}$ and $\mcT(x_1,x_2)$ are reductive. 

Now, assume that the products only consist of $u_1u_2$ or $u_1u_2u_4u_3$. Since $\Phi(u_1,u_2,u_3,u_4)$ is a non-constant symmetric polynomial, it must contain a sum term as follows:
\begin{align}
	(u_1u_2u_4u_3)^{\beta}[(u_1u_2)^{\alpha}+(u_1u_3)^{\alpha}+(u_1u_4)^{\alpha}+(u_2u_3)^{\alpha}+(u_2u_4)^{\alpha}+(u_3u_4)^{\alpha}],\notag
\end{align} where $\beta\geq 0$, $\alpha\geq 0$ and $(\beta, \alpha)\neq (0,0)$. 
\begin{itemize}[leftmargin=1em]\itemsep=3pt
\item If $F(X_1,X_2)=X_1^sX_2^t$ and $(s,t)\neq (0,0)$, then $(u_1u_2u_4u_3)^{\beta}|_{F}$ is a constant as the following:
\begin{align}
	(u_1u_2u_4u_3)^{\beta}|_{F}&=[(F^s_{1,1}F^t_{2,1})(F^s_{1,2}F^t_{2,2})(F^s_{1,4}F^t_{2,4})(F^s_{1,3}F^t_{2,3})]^{\beta}\notag\\ &=(F_{1,1}F_{1,2}F_{1,4}F_{1,3})^{s\beta}(F_{2,1}F_{2,2}F_{2,4}F_{2,3})^{t\beta}\notag\\&=4^{(s+t)\beta}.\notag
\end{align} Since $\mcT(x_1,x_2)$ is non-constant, we can assume that $\alpha\geq 1$. Note that there are two sum terms in $\sum \limits_{1\leq i< j\leq 4}(u_iu_j)^{\alpha}|_{F}$ as follows provided by $(u_1u_2)^{\alpha}$ and $(u_1u_4)^{\alpha}$: \begin{align}
	(2^sx_2^{2t})^{\alpha},\ (2^tx_1^{2s})^{\alpha}.\label{res1}
\end{align} Hence, at least one of \eqref{res1} is reductive, which implies that $\sum \limits_{1\leq i< j\leq 4}(u_iu_j)^{\alpha}|_{F}$ and $\mcT(x_1,x_2)$ are reductive by \Cref{closed}.
\item If there are at least two terms $X_1^sX_2^t$ and $X_1^pX_2^q$ in $F(X_1,X_2)$, such that $(s,t),(p,q)\neq (0,0)$ and $(s,t)\neq (p,q)$. Without loss of generality, we assume that $s> p$. According to \Cref{closed} and the discussion above, we get $\sum \limits_{1\leq i< j\leq 4}(u_iu_j)^{\alpha}|_{F}$ is reductive and only need to prove that $u_1u_2u_4u_3|_{F}$ is reductive. Note that $u_1u_2u_4u_3|_{(X_1^sX_2^t+X_1^pX_2^q)}$ is  \begin{align}
	(x_1^sx_2^t+x_1^px_2^q)\left(\frac{2^sx_2^t}{x_1^s}+\frac{2^px_2^q}{x_1^p}\right)\left(\frac{2^{s+t}}{x_1^sx_2^t}+\frac{2^{p+q}}{x_1^px_2^q}\right)\left(\frac{2^tx_1^s}{x_2^t}+\frac{2^qx_1^p}{x_2^q}\right),\notag
\end{align} and there is a sum term 
\begin{align}
	x_1^sx_2^t\cdot \frac{2^px_2^q}{x_1^p}\cdot\frac{2^{p+q}}{x_1^px_2^q}\cdot\frac{2^tx_1^s}{x_2^t}=\frac{2^{2p+q+t}x_1^{2s}}{x_1^{2p}},\notag
\end{align} which is reductive by $s>p$. Hence, by \Cref{closed}, we get $u_1u_2u_4u_3|_{F}$ is reductive. Moreover, it implies that $\mcT(x_1,x_2)$ is also reductive.
\end{itemize}
Finally, according to Step 1 and Step 2, we conclude that there are only finite positive integer solutions to the \Cref{finite equation}.
\end{proof}
Furthermore, there is a direct corollary of \Cref{finite solution} as follows.
\begin{corollary}\label{bound}
	For any positive integer solution $(x_1,x_2)$ to the \Cref{finite equation}, the inequality holds: \begin{align}\min(x_1,x_2)\leq [\mcT(a,b)]+1.\notag\end{align}
\end{corollary}
\begin{remark}
	Although it is difficult to determine whether all the solutions to \Cref{finite equation} can be generated by the initial solution $(a,b)$ through finite mutations, we can still calculate and deal with each concrete equation by \Cref{finite solution} and \Cref{bound}. Hence, we call them \emph{Diophantine equations of finite type with cluster algebraic structures}. To some degree, we partially answered the question given by \cite[Question 3.8]{CL24}.
\end{remark}
\begin{remark}\label{diophantine explanation}
Now, we compare the \Cref{finite equation} of finite type with the ones of affine type as follows. By \cite[Lemma 3.9 \& Lemma 3.12]{CL24}, there are infinite positive integer solutions to the Diophantine equations
\begin{align}
\mcT_1(x_1,x_2)=\dfrac{x_1^2+x_2^2+1}{x_1x_2}=3= \mcT_1(1,1),\notag
\end{align}
and 
\begin{align}
\mcT_2(x_1,x_2)=\dfrac{x_{2}^4+x_{1}^2+2x_{1}+1}{x_{1}x_{2}^2}=5=\mcT_2(1,1).\notag
\end{align}
Furthermore, all the solutions to them can be generated by the initial solution $(1,1)$ through the cluster mutations of $\mcA(2,-2)$ and $\mcA(1,-4)$ respectively. 

Note that in the proof of \Cref{finite solution}, we have proved that the Laurent polynomials $\mcT(x_1,x_2)$ in \Cref{finite equation} are reductive. Therefore, there are only finite solutions. However, the Laurent polynomial $\mcT_1(x_1,x_2)$ and $\mcT_2(x_1,x_2)$ are not reductive, which allows them to have infinite positive integer solutions. This phenomenon can trace back to the differences between finite type and affine (infinite) type cluster algebras of rank 2 since the Laurent mutation invariants $\mcT_1(x_1,x_2)$ and $\mcT_2(x_1,x_2)$ are not from the formula \eqref{main formula}. Hence, we provide a Diophantine explanation for the differences between cluster algebras of finite type and affine type based on \Cref{finite solution} and \cite[Lemma 3.9 \& Lemma 3.12]{CL24}.
\end{remark}
\section{Positive integer points}\label{main section}
In this section, we aim to find all the positive integer points of two mutation invariants given by Lampe \cite{Lam16}, that is both the initial cluster variables and the mutation invariants are positive integers: $(x_1,x_2,x_3,\mcT(x_1,x_2,x_3))\in \mbN_+^4$. In particular, one of them \eqref{Lampe mutation invariant 0} corresponding to the Markov equation, which is called \emph{the Markov mutation invariant} was studied by Aigner \cite{Aig13}. However, we will give a different proof by use of cluster mutations. 
\subsection{Positive integer points for the Markov mutation invariant}\

Firstly, we give a preliminary lemma by use of cluster mutations of $\mcA(2,-2)$. Recall that the corresponding cluster mutation rules are: $$\begin{array}{c}\mu_{1}(x_1,x_2)=\left(\dfrac{x_2^2+1}{x_1},x_2\right),\ \mu_{2}(x_1,x_2)=\left(x_1,\dfrac{x_1^2+1}{x_2}\right).
	\end{array}$$ By \cite[Lemma 2.23]{CL24}, there is a mutation invariant as follows
	\begin{align}
		\mathcal{T}_1(x_1,x_2)=\dfrac{x_1^2+x_2^2+1}{x_1x_2},\notag
	\end{align} that is $\mathcal{T}_1( \mu_{i}(x_1,x_2))=\mathcal{T}_1(x_1,x_2)$, for 
		$i= 1,2$. In addition, the cluster mutation $\mu_i\ (\text{for}\ i=1,2)$ keeps the integrality of any solution $(a,b)$ to $\mcT_1(x_1,x_2)=k$, where $k\in \mbN$.
\begin{lemma}\label{reduced Aigner's theorem}
	Let $k\in \mbN$ and the Diophantine equation be as follows:\begin{align}
		x_1^2+x_2^2+1=kx_1x_2. \label{reduced Markov}
	\end{align} Then, it has positive integer solutions if and only if $k=3$.
\end{lemma}
\begin{proof} 
		It is direct that when $k=3$, the \Cref{reduced Markov} has positive integer solutions. Moreover, by \cite[Lemma 3.9]{CL24}, all the solutions can be obtained by the initial solution $(1,1)$ through finite cluster mutations. Hence, we only need to prove that when $k\neq 3$, there is no positive integer solution. There are two cases for the values of $k$ to be discussed as follows.
		\begin{enumerate}[leftmargin=2em]
			\item If $k=1,2$, since $x_1^2+x_2^2\geq 2x_1x_2$, it is direct that $x_1^2+x_2^2+1>kx_1x_2$. Hence, there is no positive integer solution.  
			\item If $k\geq 4$, we assume that $(a,b)$ is a solution. If $a=1$, we have $b^2-kb+2=0$. Hence, by the quadratic formula, $\Delta=k^2-8$ must be a square number, denoted by $h^2\ (h\in \mbN_+)$. Then, we get 
			\begin{align}
				(k+h)(k-h)=8.\notag
			\end{align} Note that both $k$ and $h$ are positive integers. Then, we have 
			\begin{align}
				\left\{
		\begin{array}{ll}
			k+h=4  \\
			k-h=2
		\end{array} \right. \text{or}
		\left\{\begin{array}{ll}
			k+h=8  \\
			k-h=1
		\end{array} \right. .\notag
			\end{align} which implies that $k=3,h=1$ or $k=\frac{9}{2},h=\frac{7}{2}$. However, both of them contradict with the initial conditions. Hence, we conclude that $a\geq  2$ and $b\geq 2$. Without loss of generality, we might assume that $a\geq b\geq 2$. Let $f(\lambda)=\lambda^2-kb\lambda+b^2+1$ and  \begin{align}(a^{\prime},b)=\mu_1(a,b)=\left(\frac{b^2+1}{a},b\right).\notag 
			\end{align} We observe that $a$ and $a^{\prime}$ are the two zeros of $f(\lambda)$ and \begin{align}
				f(b)=b^2-kb^2+b^2+1=1-(k-2)b^2<0.\notag
			\end{align} It implies that $a^{\prime}<b<a$. Similarly, let $(a^{\prime},b^{\prime})=\mu_2(a^{\prime},b)$ and we have $b^{\prime}<a^{\prime}$. By repeating the cluster mutations as above, there must be a solution $(a_0,b_0)=(\mu_i\dots\mu_2\mu_1)(a,b)$ with $i\in \{1,2\}$, such that $\min(a_0,b_0)<2$. However, it contradicts with the fact that $a_0\geq b_0\geq 2$. Hence, there is no positive integer solution for $k\geq 4$. 
			\end{enumerate} 
\end{proof}
In \cite{Pro20}, Propp first investigated the relation between the Markov Diophantine equations and the once-punctured torus cluster algebra. Take the initial exchange matrix $B_{0}$  and the once-punctured torus cluster algebra is $\mcA_{P}=\mcA(B_{0})$, where\begin{align}
	B_{0}=\begin{pmatrix}0 & 2 & -2\\ -2 & 0 & 2\\ 2 & -2 & 0\end{pmatrix}.\label{2-matrix}
\end{align} Note that $B_0$ is sign-equivalent. Then, the cluster mutation rules are: $$\begin{array}{c}
	\mu_{1}(x_1,x_2,x_3)=\left(\dfrac{x_2^2+x_3^2}{x_1},x_2,x_3\right),\notag\\ \mu_{2}(x_1,x_2,x_3)=\left(x_1,\dfrac{x_1^2+x_3^2}{x_2},x_3\right),\notag\\ \mu_{3}(x_1,x_2,x_3)=\left(x_1,x_2,\dfrac{x_1^2+x_2^2}{x_3}\right).\end{array}$$ Afterwards, Lampe \cite{Lam16} exhibited a Laurent mutation invariant of $\mcA_{P}$ as follows:
\begin{align}
	\mcT_1(x_1,x_2,x_3)=\dfrac{x_1^2+x_2^2+x_3^2}{x_1x_2x_3}, \label{Lampe mutation invariant}
\end{align} that is $\mathcal{T}_1( \mu_{i}(x_1,x_2,x_3))=\mathcal{T}_1(x_1,x_2,x_3)$, for
		$i= 1,2,3$. Here, we call \eqref{Lampe mutation invariant} \emph{the Markov mutation invariant}.
		
Now, we exhibit a different proof of the Aigner's results given by \cite[Proposition 2.2]{Aig13} by use of cluster mutations.
\begin{theorem}[Aigner]\label{Aigner's theorem}
	Let $k\in \mbN$ and the Diophantine equation be as follows:
	\begin{align}
		x_1^2+x_2^2+x_3^2=kx_1x_2x_3. \label{Markov's equation}
	\end{align}
	Then, it has positive integer solutions if and only if $k=1$ or $k=3$.
\end{theorem}
\begin{proof}
	It is direct that when $k=1$ or $k=3$, the \Cref{Markov's equation} has positive integer solutions. For example, $(3,3,3)$ and $(1,1,1)$ are solutions for $k=1$ and $k=3$ respectively. Furthermore, by \cite[Theorem 2.3]{Lam16}, all the solutions for $k=3$ can be generated by the initial solution $(1,1,1)$ through finite cluster mutations $\mu_i\ (i=1,2,3)$.
	
	Hence, we only need to prove that when $k\neq 1$ and $k\neq 3$, there is no positive integer solution. There are two cases for the values of $k$ to be discussed as follows.
\begin{enumerate}[leftmargin=2em]
\item If $k=2$, we assume that $(a,b,c)$ is a solution. We claim that $\min(a,b,c)\geq 2$. Otherwise, we might assume that $a=1$. Then, we get \begin{align}
	b^2+c^2+1=2bc,\notag
\end{align} which contradicts with $b^2+c^2\geq 2bc$. Hence, without loss of generality, we assume that $a\geq b\geq c\geq 2$. Let $f(\lambda)=\lambda^2-2bc\lambda+b^2+c^2$ and \begin{align}
	(a^{\prime},b,c)=\mu_{1}(a,b,c)=\left(\dfrac{b^2+c^2}{a},b,c\right).\notag
\end{align} Then, it implies that $a$ and $a^{\prime}$ are the two zeros of $f(\lambda)$ and $a^{\prime}=2bc-a\in \mbN_+$. Note that \begin{align}
	f(b)=2b^2+c^2-2b^2c=b^2(2-c)+c(c-b^2)<0.\notag
\end{align} Hence, we conclude that $a>b>a^{\prime}$, which implies that $\max(a,b,c)>\max(a^{\prime},b,c)$.  
If $b\geq c\geq a^{\prime}$, take \begin{align}
	(a^{\prime},b^{\prime},c)=\mu_{2}(a^{\prime},b,c)=\left(a^{\prime},\frac{{a^{\prime}}^2+c^2}{b},c\right).\notag
\end{align}
Similarly, we get $b>c>b^{\prime}$, which implies that $\max(a^{\prime},b,c)> \max(a^{\prime},b^{\prime},c)$.
If $b> a^{\prime}\geq c$, take \begin{align}
	(a^{\prime},b^{\prime},c)=\mu_{2}(a^{\prime},b,c)=\left(a^{\prime},\frac{{a^{\prime}}^2+c^2}{b},c\right).\notag
\end{align} Then, we have $b>a^{\prime}>b^{\prime}$, which implies that $\max(a^{\prime},b,c)> \max(a^{\prime},b^{\prime},c)$.
 By repeating the process above, that is cluster-mutating a solution at the direction where the number is maximal, we get a sequence of solutions whose maximal numbers are strictly decreasing. However, it will end up with a solution $(a_0,b_0,c_0)$, such that 
\begin{align} 
\min(a_0,b_0,c_0)<2.\notag
\end{align} Hence, it contradicts with the fact that $\min(a_0,b_0,c_0)\geq 2$.
\item If $k\geq 4$, we assume that $(a,b,c)$ is a solution. We claim that $\min(a,b,c)\geq 2$. Otherwise, we might assume that $a=1$. Then we get \begin{align}
	b^2+c^2+1=kbc,\notag
\end{align} which contradicts with \Cref{reduced Aigner's theorem}. Hence, without loss of generality, we assume that $a\geq b\geq c\geq 2$. Let $g(\lambda)=\lambda^2-kbc\lambda+b^2+c^2$ and \begin{align}
	(a^{\prime},b,c)=\mu_{1}(a,b,c)=\left(\dfrac{b^2+c^2}{a},b,c\right).\notag
\end{align} Then, we have $a$ and $a^{\prime}$ are the two zeros of $g(\lambda)$ and $a^{\prime}=kbc-a\in \mbN_+$. Since 
\begin{align}
	g(b)=2b^2+c^2-kb^2c=2b^2(1-c)+c(c-b^2)-(k-3)b^2c<0,\notag
\end{align} we conclude that $b$ must lie between $a$ and $a^{\prime}$, that is $a>b>a^{\prime}$. Then, we get \begin{align}
	\max(a,b,c)>\max(a^{\prime},b,c). \notag
\end{align} Similar to the process of $(1)$, we can always mutate a solution at the direction where the number is maximal and get a sequence of solutions whose maximal numbers are strictly decreasing. However, it will end up with a solution $(a_0,b_0,c_0)$, such that 
\begin{align} 
\min(a_0,b_0,c_0)<2,\notag
\end{align} which is a contradiction. 
\end{enumerate}
Consequently, by cases $(1)$ and $(2)$, the theorem holds.
\end{proof}
\begin{remark} When $k=3$, all the positive integer solutions to the \Cref{Markov's equation} are called the \emph{Markov triples} and the \Cref{Markov's equation} is called the \emph{Markov equation}. In addition, for $k=1$, the positive integer triple $(a_1,a_2,a_3)$ is a solution if and only if $a_i\ (i=1,2,3)$ is divisible by $3$ and $(\frac{a_1}{3},\frac{a_2}{3},\frac{a_3}{3})$ is a solution to the Markov equation.
	However, to prove \Cref{Aigner's theorem}, Aigner \cite[Proposition 2.2]{Aig13} used the fact that the numbers in a Markov triple are pairwise relatively prime, which is different from the method of cluster mutations above.
\end{remark}
\subsection{Positive integer points for the variant of Markov mutation invariant}\

In this subsection, we will provide an analogue of Aigner's results. We aim to find all the positive integer points of \eqref{Lampe mutation invariant 1}, which is called \emph{the variant of Markov mutation invariant}. Beforehand, we need several preliminary and necessary lemmas. Note that the following lemma is direct.
\begin{lemma}\label{modulo 3}
	Let $x\in \mbN$. Then, the following properties hold: \begin{enumerate}
		\item Either $x^2\equiv0\, (\mod3)$ or $x^2\equiv1\, (\mod3)$;
		\item Either $x^2\equiv0\, (\mod4)$ or $x^2\equiv1\, (\mod4)$;
		\item Either $x^2\equiv0\, (\mod5)$ or $x^2\equiv1\, (\mod5)$ or $x^2\equiv4\, (\mod5)$;
		\item Either $x^2\equiv0\, (\mod6)$ or $x^2\equiv1\, (\mod6)$ or $x^2\equiv3\, (\mod6)$ or $x^2\equiv4\, (\mod6)$.
	\end{enumerate}
\end{lemma}
Recall that the cluster mutation rules for the cluster algebra $\mcA(1,-4)$ are $$\begin{array}{c}\mu_{1}(x_1,x_2)=\left(\dfrac{x_2^4+1}{x_1},x_2\right),\ \mu_{2}(x_1,x_2)=\left(x_1,\dfrac{x_1+1}{x_2}\right).
	\end{array}$$ It was proved in \cite[Lemma 2.26]{CL24} that there is a Laurent mutation invariant of $\mcA(1,-4)$ as follows:
	\begin{align}
\mcT_2(x_1,x_2)=\dfrac{x_2^4+x_1^2+2x_1+1}{x_1x_2^2},\notag
	\end{align} that is $\mathcal{T}_2( \mu_{i}(x_1,x_2))=\mathcal{T}_2(x_1,x_2)$, for 
		$i= 1,2$. In addition, the cluster mutation $\mu_i\ (\text{for}\ i=1,2)$ keeps the integrality of any positive integer solution $(a,b)$ to \begin{align}x_2^4+x_1^2+2x_1+1=kx_1x_2^2,\label{14}\end{align} where $k\in \mbN$.
		
\begin{lemma}\label{lemma for reduced}
	Assume that  $(a,b)$ is a positive integer solution to the \Cref{14}, where $k\geq 6$, $a\neq 1$ and $b\neq 1$.
	\begin{enumerate}[leftmargin=2em]
		\item Let $(a^{\prime},b)=\mu_1(a,b)$. If $a>b^2$, then $a>b^2>a^{\prime}$; if $a<b^2$, then $a^{\prime}>b^2>a$.
		\item Let $(a,b^{\prime})=\mu_2(a,b)$. If $a>b^2$, then ${b^{\prime}}^2>a>b^2$; if $a<b^2$, then $b^2>a>{b^{\prime}}^2$.
	\end{enumerate}
\end{lemma}
\begin{proof}
	Firstly,  we focus on the case that $(a^{\prime},b)=\mu_{1}(a,b)$, which implies that $aa^{\prime}=b^4+1$. If $a>b^2$, we claim that $a^{\prime}<b^2$. Otherwise, we obtain that $a^{\prime}\geq b^2>1$, which implies that \begin{align}aa^{\prime}\geq (b^2+1)b^2=b^4+b^2>b^4+1.\notag\end{align} Then, it is a contradiction and we have $a>b^2>a^{\prime}$. If $a<b^2$, then we get \begin{align}
		a^{\prime}=\dfrac{b^4+1}{a}>\dfrac{b^4+1}{b^2}>b^2,\notag
	\end{align} which implies that $a^{\prime}>b^2>a$.
	
	Now, we consider the second case that $(a,b^{\prime})=\mu_{2}(a,b)$, which implies that $bb^{\prime}=a+1$. If $a>b^2$, we have \begin{align}{b^{\prime}}^2a>{b^{\prime}}^2b^2=(a+1)^2.\notag\end{align}
Hence, we conclude that 
\begin{align}
	{b^{\prime}}^2>\dfrac{(a+1)^2}{a}>a>b^2.\notag
\end{align} If $a<b^2$, then we have \begin{align}
	b^{\prime}=\dfrac{a+1}{b}<\dfrac{b^2+1}{b}=b+\dfrac{1}{b},\notag
\end{align} which implies that  $b^{\prime} \leq b$. Now, we claim that $b^{\prime}< b$. Otherwise, assume that $b^{\prime}=b$ and we have $b^2$ and ${b^{\prime}}^2$ are the two common zeros of the polynomial 
\begin{align}
	f(\lambda)=\lambda^2-ka\lambda+(a+1)^2.\notag
\end{align} Then, we get
\begin{align}
	\Delta=k^2a^2-4(a+1)^2=(k^2-4)a^2-8a-4=0,\notag
\end{align} which implies that \begin{align}
	a_1=\frac{2}{k-2},\ a_2=-\frac{2}{k+2}.\notag
\end{align} Since $k\geq 6$, it implies that neither $a_1$ nor $a_2$ is a positive integer, which is a contradiction. Hence, by the Vieta's formula, we obtain that \begin{align}
	(a+1)^2=b^2{b^{\prime}}^2>{b^{\prime}}^4,\notag
\end{align} which implies that $a\geq {b^{\prime}}^2$. If ${b^{\prime}}^2= a$, then we have \begin{align}
	b^2=ka-{b^{\prime}}^2=(k-1)a. \notag
\end{align} Therefore, we obtain that 
\begin{align}
	b^2{b^{\prime}}^2=(k-1)a^2=(a+1)^2,\notag
\end{align} which implies that $k=5$ and $a=1$. However, it contradicts with $k\geq 6$ and $a\neq 1$. Hence, we have $b^2>a>{b^{\prime}}^2$.
\end{proof}
\begin{lemma}\label{reduced analogue}
	Let $k\in \mbN$ and the Diophantine equation be as follows:
	\begin{align}
	x_2^4+x_1^2+2x_1+1=kx_1x_2^2. \label{Chen's equation}
	\end{align}
	Then, it has positive integer solutions if and only if $k=5$.
\end{lemma}
\begin{proof} Firstly, it is direct that when $k=5$, the \Cref{reduced Markov} has positive integer solutions. Moreover, by \cite[Lemma 3.12]{CL24}, all the solutions can be obtained by the initial solution $(1,1)$ through finite cluster mutations. Hence, we only need to prove that when $k\neq 5$, there is no positive integer solution. Several cases need to be discussed as follows.
\begin{enumerate}[leftmargin=2em]
	\item If $k=1$, then $x_2^4-x_1x_2^2+(x_1+1)^2=0$. Regarding $x_1$ as the coefficients, we have 
	\begin{align}
		\Delta=x_1^2-4(x_1+1)^2<0,\notag
	\end{align} which means that there is no positive integer solution.
	\item If $k=2$, then $x_2^4-2x_1x_2^2+(x_1+1)^2=0$. Regarding $x_1$ as the coefficients, we have 
	\begin{align}
		\Delta=4x_1^2-4(x_1+1)^2<0,\notag
	\end{align} which means that there is no positive integer solution.
	\item If $k=3$, then $x_2^4-3x_1x_2^2+(x_1+1)^2=0$. Assume that $x_2^{(1)}$ and $x_2^{(2)}$ are the two positive integer solutions. By the Vieta's formula, we have \begin{align}
		\left\{
		\begin{array}{ll}
			\big(x_2^{(1)}\big)^2+\big(x_2^{(2)}\big)^2=3x_1 \\
\big(x_2^{(1)}\big)^2\cdot\big(x_2^{(2)}\big)^2=\big(x_1+1\big)^2
		\end{array} \right. .\notag
	\end{align} Hence, we get $(x_2^{(1)}-x_2^{(2)})^2=x_1-2$ and denote by $x_1=s^2+2$, where $s=|x_2^{(1)}-x_2^{(2)}|\in \mbN$. Note that \begin{align}
		\Delta&=9x_1^2-4(x_1+1)^2\notag \\\notag &=5x_1^2-8x_1-4\\&=s^2(5s^2+12)\notag
	\end{align} is a square number, which implies that $5s^2+12$ is also a square number. Assume that there is a non-negative integer $h$ such that 
	\begin{align}
		5s^2+12=h^2.\notag
	\end{align} By \Cref{modulo 3} $(1)$, we have $s^2\equiv0\,(\mod 3)$ and $h^2\equiv0\,(\mod 3)$, which implies that $s\equiv0\,(\mod 3)$ and $h\equiv0\,(\mod 3)$. However, note that \begin{align}
		5s^2+12\equiv h^2\equiv 0\,(\mod9)\notag
	\end{align} and it contradicts with $5s^2+12\nequiv 0\,(\mod9)$. Therefore, there is no positive integer solution for $k=3$.
	\item If $k=4$, then $x_2^4-4x_1x_2^2+(x_1+1)^2=0$. Assume that $x_2^{(1)}$ and $x_2^{(2)}$ are the two positive integer solutions. By the Vieta's formula, we have \begin{align}
		\left\{
		\begin{array}{ll}
			\big(x_2^{(1)}\big)^2+\big(x_2^{(2)}\big)^2=4x_1  \\
			\big(x_2^{(1)}\big)^2\cdot\big(x_2^{(2)}\big)^2=\big(x_1+1\big)^2
		\end{array} \right. .\notag
	\end{align} Hence, we get $(x_2^{(1)}-x_2^{(2)})^2=2(x_1-1)$ is a square number. Assume that $x_1=2s^2+1$ and we get 
	\begin{align}
		\Delta&=16x_1^2-4(x_1+1)^2\notag \\&=4(3x_1^2-2x_1-1)\notag \\&=16s^2(3s^2+2).\notag
	\end{align} It implies that $3s^2+2$ is a square number. Hence, we assume that there is a positive integer $h$, such that
	\begin{align}
		3s^2+2=h^2.\notag
	\end{align}
	However, by \Cref{modulo 3} $(1)$, it is a contradiction. Then, there is no positive integer solution for $k=4$.
	\item If $k\geq 6$, then $x_2^4-kx_1x_2^2+(x_1+1)^2=0$. Assume that $(a,b)$ is a positive integer solution. Then, we claim that $a\neq b^2$. Otherwise, we have 
	\begin{align}
		(k-2)a^2-2a-1=0,\notag
	\end{align} which implies that \begin{align}a=\dfrac{1\pm\sqrt{k-1}}{k-2}<1.\notag\end{align} Hence, it is a contradiction. By \Cref{lemma for reduced}, after finite cluster mutations, we must get a solution $(a_0,b_0)$, such that $a_0=1$ or $b_0=1$. 
	\begin{itemize}[leftmargin=1em]\itemsep=0pt	
\item If $a_0=1$, then $b_{0}^4-kb_{0}^2+4=0$. Then, it implies that $\Delta=k^2-16$ must be a square number. Assume that $k^2-16=m^2\ (m\in \mbN)$. Then, we have $(k-m)(k+m)=16$, which implies that \begin{align}
				\left\{
		\begin{array}{ll}
			k-m=1  \\
			k+m=16
		\end{array} \right. \text{or}\,
		\left\{\begin{array}{ll}
			k-m=2  \\
			k+m=8
		\end{array} \right. \text{or}\,
		\left\{\begin{array}{ll}
			k-m=4  \\
			k+m=4
		\end{array} \right. .\notag
			\end{align}
	Moreover, we have 
			\begin{align}
				\left\{
		\begin{array}{ll}
			k=\frac{17}{2}  \\
			m=\frac{15}{2}
		\end{array} \right. \text{or}\,
		\left\{\begin{array}{ll}
			k=5  \\
			m=3
		\end{array} \right. \text{or}\,
		\left\{\begin{array}{ll}
			k=4  \\
			m=0
		\end{array} \right. .\notag
			\end{align} which contradicts with $k\geq 6$.
		\item If $b_0=1$, then $a_{0}^2+(2-k)a_{0}+2=0$. Then, it implies that $\Delta=k^2-4k-4$ must be a square number. Assume that $k^2-4k-4=n^2\ (n\in \mbN)$. Then, we have $(k-2-n)(k-2+n)=8$, which implies that \begin{align}
				\left\{
		\begin{array}{ll}
			k-2-n=1  \\
			k-2+n=8
		\end{array} \right. \text{or}\,
		\left\{\begin{array}{ll}
			k-2-n=2  \\
			k-2+n=4
		\end{array} \right. .\notag
			\end{align}
	Moreover, we have 
			\begin{align}
				\left\{
		\begin{array}{ll}
			k=\frac{13}{2}  \\
			n=\frac{7}{2}
		\end{array} \right. \text{or}\,
		\left\{\begin{array}{ll}
			k=5  \\
			n=1
		\end{array} \right. ,\notag
			\end{align} which contradicts with $k\geq 6$.
\end{itemize}
Hence, there is no positive integer solution for $k\geq 6$. 
	\end{enumerate}
In conclusion, the \Cref{Chen's equation} has positive integer solutions if and only if $k=5$.
\end{proof}
In \cite[Proposition 2.5]{Lam16}, Lampe studied a special but important finite mutation type cluster algebra $\mcA_{L}=\mcA(B_0)$, whose initial exchange matrix is \begin{align}
	B_0=\begin{pmatrix}0 & 1 & -1\\ -4 & 0 & 2\\ 4 & -2 & 0\end{pmatrix}.\label{4-matrix}
\end{align} Note that $B_0$ is sign-equivalent and the cluster mutation rules are: 
$$\begin{array}{c}
	\mu_{1}(x_1,x_2,x_3)=\left(\dfrac{x_2^4+x_3^4}{x_1},x_2,x_3\right),\\ \mu_{2}(x_1,x_2,x_3)=\left(x_1,\dfrac{x_1+x_3^2}{x_2},x_3\right),\\ \mu_{3}(x_1,x_2,x_3)=\left(x_1,x_2,\dfrac{x_1+x_2^2}{x_3}\right).
\end{array}$$
Then, Lampe put forward a Laurent mutation invariant of $\mcA_{L}$ as follows: \begin{align}
\mcT_2(x_1,x_2,x_3)=\dfrac{x_1^2+x_2^4+x_3^4+2x_1x_2^2+2x_1x_3^2}{x_1x_2^2x_3^2},\label{Lampe mutation invariant 3}
\end{align} that is $\mathcal{T}_2( \mu_{i}(x_1,x_2,x_3))=\mathcal{T}_2(x_1,x_2,x_3)$, for 
		$i= 1,2,3$. Here, we call \eqref{Lampe mutation invariant 3} \emph{the variant of Markov mutation invariant}. According to \cite[Theorem 2.6]{Lam16}, all
		 the solutions to a variant of Markov Diophantine equations as follows \begin{align}
x_1^2+x_2^4+x_3^4+2x_1x_2^2+2x_1x_3^2=7x_1x_2^2x_3^2\notag
		 \end{align} can be generated by the initial solution $(1,1,1)$ through finite cluster mutations.

Hence, motivated by Aigner's work \cite{Aig13}, a natural question is to find the conditions for the equations $\mcT_2(x_1,x_2,x_3)=k\ (k\in \mbN)$ to have positive integer solutions. Equivalently, we need to find all the positive integer points of $\mathcal{T}_2(x_1,x_2,x_3)$.

\begin{theorem}\label{analogue}
	Let $k\in \mbN$ and the Diophantine equation be as follows:
	\begin{align}
x_1^2+x_2^4+x_3^4+2x_1x_2^2+2x_1x_3^2=kx_1x_2^2x_3^2. \label{Lampe's equation}
	\end{align}
	Then, it has positive integer solutions if and only if $k=7$.
\end{theorem}
\begin{proof}
	By \cite[Theorem 2.6]{Lam16}, it is direct that when $k=7$, there are positive integer solutions to the \Cref{Lampe's equation}. Hence, we only need to prove that when $1\leq k\leq 6$ and  $k\geq 8$, there is no positive integer solution. Let $X_1=x_1, X_2=x_2^2,X_3=x_3^2$. Then, we have the \Cref{Lampe's equation} is equivalent to 
	\begin{align}
X_1^2+X_2^2+X_3^2+2X_1X_2+2X_1X_3=kX_1X_2X_3. \label{Lampe's equation II}
	\end{align} Without loss of generality, we assume that $(a,b,c)$ is a positive integer solution to the \Cref{Lampe's equation} with $b\geq c$, which implies that $(A,B,C)=(a,b^2,c^2)$ is a positive integer solution to the \Cref{Lampe's equation II} with $B\geq C$. By \Cref{reduced analogue}, we have $b\neq 1$ and $c\neq 1$, which implies that $B=b^2\geq 4$ and $C=c^2\geq 4$.
	
	 Now, we define a variant of cluster mutation maps $\widehat{\mu_i}: \mbQ_+^3\rightarrow \mbQ_+^3$ by $$\begin{array}{c}
	\widehat{\mu_1}(X_1,X_2,X_3)=\left(\dfrac{X_2^2+X_3^2}{X_1},X_2,X_3\right),\\ \widehat{\mu_2}(X_1,X_2,X_3)=\left(X_1,\dfrac{(X_1+X_3)^2}{X_2},X_3\right),\\ \widehat{\mu_3}(X_1,X_2,X_3)=\left(X_1,X_2,\dfrac{(X_1+X_2)^2}{X_3}\right).
\end{array}$$ Let $S$ be a map from $\mbQ_+^3$ to itself, such that $S(x_1,x_2,x_3)=(x_1,x_2^2,x_3^2)$. Then, we have $\mu_i$ and $\widehat{\mu_i}$ are compatible with $S$, that is \begin{align}\widehat{\mu_i}(S(x_1,x_2,x_3))=S(\mu_i(x_1,x_2,x_3)),\ \text{for}\ i= 1,2,3.\notag\end{align}
Let $f(\lambda)=\lambda^2+(2B+2C-kBC)\lambda+B^2+C^2$ and $g(\lambda)=\lambda^2+(2A-kAC)\lambda+(A+C)^2.$	
Note that $A$ is a zero of $f(\lambda)$ and $B$ is a zero of $g(\lambda)$. Take $(A^{\prime},B,C)=\widehat{\mu_1}(A,B,C)$ and $(A,B^{\prime},C)=\widehat{\mu_2}(A,B,C)$. By the Vieta's formula, $A^{\prime}$ is another zero of $f(\lambda)$ and $B^{\prime}$ is another zero of $g(\lambda)$. Furthermore, according to the mutation rules and the equalities
\begin{align}
	A+A^{\prime}=kBC-2B-2C,\ B+B^{\prime}=kAC-2A,\notag
\end{align} we conclude that both $A^{\prime}$ and $B^{\prime}$ are positive integers.

Now, we claim that after mutating a solution at the direction where the component is maximal, we can obtain a new solution whose maximal component is strictly less than before. There are several cases for the values of $k$ to be discussed as follows.
	\begin{enumerate}[leftmargin=2em]
		\item If $k$=1, then we have $a^2+b^4+c^4+2ab^2+2ac^2=ab^2c^2$ and $A^2+B^2+C^2+2AB+2AC=ABC$. Note that $a\geq3, b\geq 3, c\geq 3$, which implies that \begin{align}A\geq 3,B\geq 9,C\geq 9.\notag\end{align} 
		Beforehand, we claim that $A\geq 7$. In fact,\begin{itemize}[leftmargin=1em]\itemsep=0pt
 	\item If $A=a=3$, we have $9+b^4+c^4+6b^2+6c^2=3b^2c^2$, which implies that $b^4+c^4\equiv 0\, (\mod 3)$. By \Cref{modulo 3}, we have $b^2\equiv 0\, (\mod 3)$ and $c^2\equiv 0\, (\mod 3)$. Hence, we get $b\equiv 0\, (\mod 3)$ and $c\equiv 0\, (\mod 3)$. Note that \begin{align}9+b^4+c^4+6b^2+6c^2\nequiv 0\, (\mod 27),\notag\end{align} which contradicts with $3b^2c^2\equiv 0\, (\mod 27)$.
 	\item If $A=a=4$, we have $16+b^4+c^4+8b^2+8c^2=4b^2c^2$. 
 	Note that by \Cref{modulo 3}, we have $b^2\equiv 1\, (\mod 3)$ and $c^2\equiv 1\, (\mod 3)$. Assume that $b^2=3s+1$ and $c^2=3t+1$, where $s,t \in \mbN$. Then, the equation is equivalent to \begin{align}
 	9s^2+9t^2+18s+18t+30&=36st.\notag 
 	\end{align} However, note that $9s^2+9t^2+18s+18t+30\nequiv 0\, (\mod 9)$, which contradicts with $36st\equiv 0\, (\mod 9)$.

 	\item If $A=a=5$, we have $25+b^4+c^4+10b^2+10c^2=5b^2c^2$, which implies that $b^4+c^4\equiv 0\, (\mod 5)$. By \Cref{modulo 3}, we have $b^2\equiv 0\, (\mod 5)$ and $c^2\equiv 0\, (\mod 5)$. Hence, we get $b\equiv 0\, (\mod 5)$ and $c\equiv 0\, (\mod 5)$. Note that \begin{align}25+b^4+c^4+10b^2+10c^2\nequiv 0\, (\mod 125),\notag\end{align} which contradicts with $5b^2c^2\equiv 0\, (\mod 125)$.
 	\item If $A=a=6$, we have $36+b^4+c^4+12b^2+12c^2=6b^2c^2$, which implies that $b^4+c^4\equiv 0\, (\mod 6)$. By \Cref{modulo 3}, we have \begin{align}
 		\left\{
		\begin{array}{ll}
			b^2\equiv 0\, (\mod 6)  \\
			c^2\equiv 0\, (\mod 6)
		\end{array} \right. \text{or}\,
		\left\{\begin{array}{ll}
			b^2\equiv 3\, (\mod 6)  \\
			c^2\equiv 3\, (\mod 6)
		\end{array} \right. .\notag
 	\end{align} The former implies that $b\equiv 0\, (\mod 6)$ and $c\equiv 0\, (\mod 6)$. However, note that \begin{align}36+b^4+c^4+12b^2+12c^2\nequiv 0\, (\mod 216),\notag\end{align} which contradicts with $6b^2c^2\equiv 0\, (\mod 216)$. As for the later, assume that $b^2=6s+3$ and $c^2=6t+3$, where $s,t \in \mbN$. Then, the equation is equivalent to \begin{align}
 		s^2+t^2+2=6st,\notag
 	\end{align} which implies that $s$ and $t$ are both odd or even. When $s=2j$ and $t=2k$, where $j,k\in \mbN$, we have \begin{align}
 		4j^2+4k^2+2=24jk.\notag
 	\end{align} Note that $4j^2+4k^2+2\nequiv 0\, (\mod 4)$, which contradicts with $24jk \equiv 0\, (\mod 4)$. When $s=2j+1$ and $t=2k+1$, where $j,k\in \mbN$, we have \begin{align}
 	4j^2+4j+4k^2+4k+4=24jk+12j+12k+6.\notag
 	\end{align} Note that $4j^2+4j+4k^2+4k+4\equiv 0\, (\mod 4)$, which contradicts with $24jk+12j+12k+6 \nequiv 0\, (\mod 4)$.
 \end{itemize}
Hence, we obtain that $A=a\geq 7$.

		Now, we firstly assume that $A\geq B$. Then, we have $\max(A,B,C)=A$ and 
		\begin{align}
			f(B)&=4B^2+C^2+2BC-B^2C\notag\\&=B^2(9-C)-5B^2+2BC+C^2\notag\\ &\leq B^2(9-C)+(C^2-3B^2)\notag\\&<0.\notag
		\end{align} Hence, we get $B$ lies between $A$ and $A^{\prime}$, which implies that $A^{\prime}<B<A$. Therefore, we have \begin{align}
	\max(A^{\prime},B,C)<\max(A,B,C).\notag
		\end{align} 
		Secondly, we assume that $A<B$. If $C\leq A$, then we have $\max(A,B,C)=B$ and 
		\begin{align}
			g(A)&=4A^2-A^2C+2AC+C^2\notag\\ &=A^2(9-C)-5A^2+2AC+C^2\notag\\&\leq A^2(9-C)+(C^2-3A^2)\notag\\ &<0.\notag
		\end{align} It implies that $B^{\prime}<A<B$ and $C<B$. Therefore, we have \begin{align} 
	\max(A,B^{\prime},C)<\max(A,B,C).\notag
		\end{align} If $A< C$, then we have $\max(A,B,C)=B$ and \begin{align} 
			g(C)&=2C^2-AC^2+4AC+A^2\notag\\&= C^2(7-A)-5C^2+4AC+A^2\notag\\ &\leq C^2(7-A)+(A^2-C^2)\notag\\ &<0.\notag
		\end{align} It implies that $B^{\prime}<C<B$ and $A<B$. Therefore, we have \begin{align} 
	\max(A,B^{\prime},C)<\max(A,B,C).\notag
		\end{align}
		\item If $k=2$, then we have $a^2+b^4+c^4+2ab^2+2ac^2=2ab^2c^2$ and $A^2+B^2+C^2+2AB+2AC=2ABC$. It is direct that $a=A\geq 2$ and we claim that $A\geq 4$. In fact,\begin{itemize}[leftmargin=1em]\itemsep=0pt
 	\item If $A=a=2$, we have $4+b^4+c^4+4b^2+4c^2=4b^2c^2$, which implies that $b^4+c^4\equiv 0\, (\mod 4)$. By \Cref{modulo 3}, we have $b^2\equiv 0\, (\mod 4)$ and $c^2\equiv 0\, (\mod 4)$. Hence, we get $b\equiv 0\, (\mod 2)$ and $c\equiv 0\, (\mod 2)$. Note that \begin{align}4+b^4+c^4+4b^2+4c^2\nequiv 0\, (\mod 16),\notag\end{align} which contradicts with $4b^2c^2\equiv 0\, (\mod 16)$.
 	\item If $A=a=3$, we have $9+b^4+c^4+6b^2+6c^2=6b^2c^2$, which implies that $b^4+c^4\equiv 0\, (\mod 3)$. By \Cref{modulo 3}, we have $b^2\equiv 0\, (\mod 3)$ and $c^2\equiv 0\, (\mod 3)$. Hence, we get $b\equiv 0\, (\mod 3)$ and $c\equiv 0\, (\mod 3)$. Note that \begin{align}9+b^4+c^4+6b^2+6c^2\nequiv 0\, (\mod 27),\notag\end{align} which contradicts with $6b^2c^2\equiv 0\, (\mod 27)$.
 	\end{itemize}
 Hence, we have proved that $A=a\geq 4$.
 	
		 Firstly, we assume that $A\geq B$. Then, we have $\max(A,B,C)=A$ and 
		\begin{align}
			f(B)&=4B^2+C^2+2BC-2B^2C\notag\\&=2B^2(4-C)-4B^2+C^2+2BC\notag\\ &\leq 2B^2(4-C)-2B^2+C^2\notag\\&<0.\notag
		\end{align} Hence, we get $B$ lies between $A$ and $A^{\prime}$, which implies that $A^{\prime}<B<A$. Therefore, we have \begin{align}
	\max(A^{\prime},B,C)<\max(A,B,C).\notag
		\end{align} Secondly, we assume that $A<B$. If $C\leq A$, then we have $\max(A,B,C)=B$ and 
		\begin{align}
			g(A)&=4A^2-2A^2C+2AC+C^2\notag\\ &=2A^2(4-C)-4A^2+2AC+C^2\notag\\&\leq 2A^2(4-C)+(C^2-2A^2)\notag\\ &<0.\notag
		\end{align} It implies that $B^{\prime}<A<B$ and $C<B$. Therefore, we have \begin{align} 
	\max(A,B^{\prime},C)<\max(A,B,C).\notag
		\end{align} If $A< C$, then we have $\max(A,B,C)=B$ and \begin{align} 
			g(C)&=2C^2-2AC^2+4AC+A^2\notag\\&= 2C^2(4-A)-6C^2+4AC+A^2\notag\\ &\leq 2C^2(4-A)+(A^2-2C^2)\notag\\ &<0.\notag
		\end{align} It implies that $B^{\prime}<C<B$ and $A<B$. Therefore, we have \begin{align} 
	\max(A,B^{\prime},C)<\max(A,B,C).\notag
		\end{align}
		\item If $k=3$, then we have $a^2+b^4+c^4+2ab^2+2ac^2=3ab^2c^2$ and $A^2+B^2+C^2+2AB+2AC=3ABC$. Note that $a\geq 2$. Otherwise, we have \begin{align}
	b^4+c^4+2b^2+2c^2+1=3b^2c^2,\notag
		\end{align} which implies that $b^4+c^4+2b^2+2c^2+1\equiv 0\, (\mod 3)$. However, by \Cref{modulo 3}, it implies that $b^4+c^4+2b^2+2c^2+1\equiv 1\, (\mod 3)$, which is a contradiction. Hence, we have $A=a\geq 2$.
		
		 Firstly, we assume that $A\geq B$. Then, we have $\max(A,B,C)=A$ and 
		\begin{align}
			f(B)&=4B^2+C^2+2BC-3B^2C\notag\\&=3B^2(4-C)-8B^2+C^2+2BC\notag\\ &\leq 3B^2(4-C)+(C^2-6B^2)\notag\\&<0.\notag
		\end{align} Hence, we get $B$ lies between $A$ and $A^{\prime}$, which implies that $A^{\prime}<B<A$. Therefore, we have \begin{align}
	\max(A^{\prime},B,C)<\max(A,B,C).\notag
		\end{align} Secondly, we assume that $A<B$. If $C\leq A$, then we have $\max(A,B,C)=B$ and 
		\begin{align}
			g(A)&=4A^2-3A^2C+2AC+C^2\notag\\ &=3A^2(4-C)-8A^2+2AC+C^2\notag\\&\leq 3A^2(4-C)+(C^2-6A^2)\notag\\ &<0.\notag
		\end{align} It implies that $B^{\prime}<A<B$ and $C<B$. Therefore, we have \begin{align} 
	\max(A,B^{\prime},C)<\max(A,B,C).\notag
		\end{align} If $A< C$, then we have $\max(A,B,C)=B$ and \begin{align} 
			g(C)&=2C^2-3AC^2+4AC+A^2\notag\\&= 2C^2(2-A)-2C^2-AC^2+4AC+A^2\notag\\ &= 2C^2(2-A)+AC(4-C)+(A^2-2C^2)\notag\\ &<0.\notag
		\end{align} It implies that $B^{\prime}<C<B$ and $A<B$. Therefore, we have \begin{align} 
	\max(A,B^{\prime},C)<\max(A,B,C).\notag
		\end{align}

		\item If $k=4,5,6$ or $k\geq 8$, we have $A^2+B^2+C^2+2AB+2AC=kABC$. Firstly, we assume that $A\geq B$. Then, we have $\max(A,B,C)=A$ and 
		\begin{align}
			f(B)&=4B^2+C^2+2BC-kB^2C\notag\\&=4B^2(4-C)-12B^2+C^2+2BC+(4-k)B^2C\notag\\ &\leq 4B^2(4-C)+(C^2-10B^2)+(4-k)B^2C\notag\\&<0.\notag
		\end{align} Hence, we get $B$ lies between $A$ and $A^{\prime}$, which implies that $A^{\prime}<B<A$. Therefore, we have \begin{align}
	\max(A^{\prime},B,C)<\max(A,B,C).\notag
		\end{align} Secondly, we assume that $A<B$. If $C\leq A$, then we have $\max(A,B,C)=B$ and 
		\begin{align}
			g(A)&=4A^2-4A^2C+2AC+C^2+(4-k)A^2C\notag\\ &=4A^2(4-C)-12A^2+2AC+C^2+(4-k)A^2C\notag\\&\leq 4A^2(4-C)+(C^2-10A^2)+(4-k)A^2C\notag\\ &<0.\notag
		\end{align} It implies that $B^{\prime}<A<B$ and $C<B$. Therefore, we have \begin{align} 
	\max(A,B^{\prime},C)<\max(A,B,C).\notag
		\end{align} If $A< C$, then we have $\max(A,B,C)=B$ and \begin{align} 
			g(C)&=2C^2-4AC^2+4AC+A^2+(4-k)AC^2\notag\\ &= 2C^2(1-A)+2AC(4-C)+A(A-4C)+(4-k)AC^2\notag\\ &<0.\notag
		\end{align} It implies that $B^{\prime}<C<B$ and $A<B$. Therefore, we have \begin{align} 
	\max(A,B^{\prime},C)<\max(A,B,C).\notag
		\end{align}
	\end{enumerate}
Hence, the claim holds, which means the certain cluster mutations $\widehat{\mu_i}$ can decrease the maximal component of a solution constantly. Then, note that \begin{enumerate}[leftmargin=2em]
\item When $k=1$, the process will end up with a solution $(A_0,B_0,C_0)$, such that \begin{align}
	A_0< 7\ \text{or}\ \min(B_0,C_0)<4.\notag 
\end{align}
\item When $k=2$, the process will end up with a solution $(A_0,B_0,C_0)$, such that \begin{align}
	A_0< 4\ \text{or}\ \min(B_0,C_0)<4. \notag
\end{align}
\item When $k=3$, the process will end up with a solution $(A_0,B_0,C_0)$, such that \begin{align}
	A_0< 2\ \text{or}\ \min(B_0,C_0)<4. \notag
\end{align}
	\item When $k=4,5,6$ or $k\geq 8$, the process will end up with a solution $(A_0,B_0,C_0)$, such that \begin{align}
	\min(B_0,C_0)<4. \notag
\end{align}
\end{enumerate}
However, for all the cases that $k\neq 7$, the conditions $(1)-(4)$ lead to a contradiction with the initial conditions. Hence, we complete the proof.
 \end{proof}
\begin{remark}
	To some extent, based on \Cref{Aigner's theorem} and \Cref{analogue}, we generalize the results given by Aigner \cite{Aig13} and Lampe \cite{Lam16} by use of cluster mutations. \end{remark}
\section{Application: Diophantine equations with cluster algebraic structures}\label{S6}
In this section, as an application of the results in \Cref{main section}, we aim to find several classes of Diophantine equations with cluster algebraic structures, that is all the positive integer solutions can be generated by an initial solution through finite cluster mutations.
\subsection{Diophantine equations with two variables}\

First of all, we recall an important lemma about commutative algebras as follows.
\begin{lemma}[{\cite[Example 5.0]{AM69}}]\label{Atiyah}
	The integer ring $\mbZ$ is integrally closed in the rational field $\mbQ$.
\end{lemma}
In other words, for any $r\in \mbQ$, if it is a root of a monic polynomial with coefficients in $\mbZ$, then $r\in \mbZ$.
\begin{proposition}\label{22case}
	Let $F(X)\in \mbZ[X]$ be a non-constant monic polynomial and $t\in \mbZ$. The Diophantine equation \begin{align}
		F\left(\dfrac{x_1^2+x_2^2+1}{x_1x_2}\right)=F(t) \label{F22}
	\end{align} has positive integer solutions if and only if $F(t)=F(3)$. Under this case, all the positive integer solutions can be obtained from the initial solution $(1,1)$ by finite cluster mutations of $\mcA(2,-2)$.
\end{proposition}
\begin{proof}
	Let $G(X)=F(X)-F(t)$ and it is a monic polynomial over $\mbZ$. Then, the \Cref{F22} is equivalent to 
	\begin{align}
		G\left(\dfrac{x_1^2+x_2^2+1}{x_1x_2}\right)=0.\label{G22}
	\end{align} According to \Cref{Atiyah}, the \Cref{G22} has positive integer solutions if and only if \begin{align}\frac{x_1^2+x_2^2+1}{x_1x_2}\in \mbZ_{\geq 0}.\notag\end{align} Moreover, by \Cref{reduced Aigner's theorem}, we have \begin{align}\frac{x_1^2+x_2^2+1}{x_1x_2}=3,\notag\end{align} which implies that $F(t)=F(3)$. Then, based on \cite[Lemma 3.9]{CL24}, all the positive integer solutions can be generated by the initial solution $(1,1)$ through finite cluster mutations of $\mcA(2,-2)$.
	\end{proof}
\begin{example}
	Take $F(X)=(X-3)(X-4)$ and $t=4$. Then the equation \begin{align}
		F\left(\dfrac{x_1^2+x_2^2+1}{x_1x_2}\right)=F(4)\notag
	\end{align} has positive integer solutions since $F(4)=F(3)=0$. Then, all its positive integer solutions can be obtained from the initial solution $(1,1)$ by finite cluster mutations of $\mcA(2,-2)$.
\end{example}
\begin{remark}\label{counter-example 22}
	Note that when $F(X)$ is not monic, the \Cref{22case} does not hold. We give two counter-examples as follows.
	\begin{enumerate}[leftmargin=2em]
		\item Take $F_1(X)=4X^2-17X+18$ and $t=2$. Then $F_1(2)=0\neq F_1(3)$. However, $(2,2)$ is a solution to the equation \begin{align}
			F_1\left(\dfrac{x_1^2+x_2^2+1}{x_1x_2}\right)=F_1(2),\notag
		\end{align} which can't be obtained from the initial solution $(1,1)$ by finite cluster mutations of $\mcA(2,-2)$.
		\item Take $F_2(X)=3X^2-20X+36$ and $t=3$. Then, the equation is \begin{align}
			F_2\left(\dfrac{x_1^2+x_2^2+1}{x_1x_2}\right)=F_2(3).\notag
		\end{align} Note that both $(1,3)$ and $(10,3)$ are its solutions, which can't be obtained from $(1,1)$ by finite cluster mutations of $\mcA(2,-2)$.
	\end{enumerate}
\end{remark}
	
	
	
\begin{proposition}\label{14case}
	Let $F(X)\in \mbZ[X]$ be a non-constant monic polynomial and $t\in \mbZ$. The Diophantine equation \begin{align}
		F\left(\dfrac{x_2^4+x_1^2+2x_1+1}{x_1x_2^2}\right)=F(t) \label{F14}
	\end{align} has positive integer solutions if and only if $F(t)=F(5)$. Under this case, all the positive integer solutions can be obtained from the initial solution $(1,1)$ by finite cluster mutations of $\mcA(1,-4)$.
\end{proposition}
\begin{proof}
	Let $G(X)=F(X)-F(t)$ and it is a monic polynomial over $\mbZ$. Then, the \Cref{F14} is equivalent to 
	\begin{align}
		G\left(\dfrac{x_2^4+x_1^2+2x_1+1}{x_1x_2^2}\right)=0.\label{G14}
	\end{align} According to \Cref{Atiyah}, the \Cref{G14} has positive integer solutions if and only if \begin{align}\frac{x_2^4+x_1^2+2x_1+1}{x_1x_2^2}\in \mbZ_{\geq 0}.\notag\end{align} Moreover, by \Cref{reduced analogue}, we have \begin{align}\frac{x_2^4+x_1^2+2x_1+1}{x_1x_2^2}=5,\notag\end{align} which implies that $F(t)=F(5)$. Hence, based on \cite[Lemma 3.12]{CL24}, all the positive integer solutions can be generated by the initial solution $(1,1)$ through finite cluster mutations of $\mcA(1,-4)$.
\end{proof}
	
	
	
\subsection{Diophantine equations with three variables}\

In what follows, for brevity, we always denote by $\mcA_{P}$ and $\mcA_{L}$ the cluster algebras with the initial exchange matrix \eqref{2-matrix} and \eqref{4-matrix} respectively.
\begin{lemma}\label{1Markov}
	All the positive integer solutions to the Diophantine equation 
	\begin{align}
		x_1^2+x_2^2+x_3^2=x_1x_2x_3 \label{M}
	\end{align} can be obtained from the initial solution $(3,3,3)$ by finite cluster mutations of $\mcA_P$.
\end{lemma}
\begin{proof}
	By \cite[Proposition 2.2]{Aig13}, there is a bijection between the Markov equation \eqref{Markov's equation} and the \Cref{M}. More precisely, the positive integer triple $(a_1,a_2,a_3)$ is a solution to the \Cref{M} if and only if $a_i\ (\text{for}\ i=1,2,3)$ is divisible by $3$ and $(\frac{a_1}{3},\frac{a_2}{3},\frac{a_3}{3})$ is a solution to the Markov equation \eqref{Markov's equation}. 
	
	Note that by the mutation rules of $\mcA_{P}$, we have 
	\begin{align}
		3\mu_i(x_1,x_2,x_3)=\mu_{i}(3x_1,3x_2,3x_3),\ \text{for}\ i=1,2,3.\notag
	\end{align} Hence, by induction, we obtain that  
	\begin{align}
		3(\mu_{k_t}\dots  \mu_{k_1})(x_1,x_2,x_3)=(\mu_{k_t}\dots  \mu_{k_1})(3x_1,3x_2,3x_3),\notag
	\end{align} for any $t\in \mbN$ and $(k_1,\dots k_t)\in \{1,2,3\}^t$.
In addition, by \cite[Theorem 2.3]{Lam16}, any solution $(a,b,c)$ to the Markov equation \eqref{Markov's equation} can be written as 
\begin{align}
	(a,b,c)=(\mu_{k_r}\dots  \mu_{k_1})(1,1,1),\notag
\end{align} where $r\in \mbN$ and $(k_r,\dots k_1)\in \{1,2,3\}^r$. Consequently, for any solution $(a_1,a_2,a_3)$ to the \Cref{M}, we have 
\begin{align}
	(a_1,a_2,a_3)&=3\left(\frac{a_1}{3},\frac{a_2}{3},\frac{a_3}{3}\right)\notag\\&=3(\mu_{k_r}\dots  \mu_{k_1})(1,1,1)\notag\\&=(\mu_{k_r}\dots  \mu_{k_1})(3,3,3).\notag
\end{align} Then, all the positive integer solutions to the \Cref{M} can be obtained from the initial solution $(3,3,3)$ by finite cluster mutations of $\mcA_P$.
\end{proof}
\begin{proposition}\label{Prop case}
	Let $F(X)\in \mbZ[X]$ be a non-constant monic polynomial and $t\in \mbZ$. The Diophantine equation \begin{align}
		F\left(\dfrac{x_1^2+x_2^2+x_3^2}{x_1x_2x_3}\right)=F(t) \label{F222}
	\end{align} has positive integer solutions if and only if $F(t)=F(1)$ or $F(t)=F(3)$. Moreover, \begin{enumerate}[leftmargin=2em]
		\item if $F(1)=F(3)$, all the positive integer solutions can be obtained from the initial solutions $(1,1,1)$ and $(3,3,3)$ by finite cluster mutations of $\mcA_P$;
		\item if $F(1)\neq  F(3)$, all the positive integer solutions can be obtained from either the initial solution $(1,1,1)$ or the initial solution $(3,3,3)$ by finite cluster mutations of $\mcA_P$.
	\end{enumerate} 
\end{proposition}
\begin{proof}
Let $G(X)=F(X)-F(t)$ and it is a monic polynomial over $\mbZ$. Then, the \Cref{F222} is equivalent to 
	\begin{align}
G\left(\dfrac{x_1^2+x_2^2+x_3^2}{x_1x_2x_3}\right)=0.\label{G222}
	\end{align}
	According to \Cref{Atiyah}, the \Cref{G222} has positive integer solutions if and only if \begin{align}\frac{x_1^2+x_2^2+x_3^2}{x_1x_2x_3}\in \mbZ_{\geq 0}.\notag\end{align} Then, by \Cref{Aigner's theorem}, we have \begin{align}\frac{x_1^2+x_2^2+x_3^2}{x_1x_2x_3}=1\ \text{or}\ \frac{x_1^2+x_2^2+x_3^2}{x_1x_2x_3}=3,\notag\end{align} which implies that $F(t)=F(1)$ or $F(t)=F(3)$. Furthermore, based on \Cref{1Markov} and \cite[Theorem 2.3]{Lam16}, the proposition holds for $F(1)=F(3)$ and $F(1)\neq F(3)$.

\end{proof}

\begin{proposition}\label{Lampe case}
	Let $F(X)\in \mbZ[X]$ be a non-constant monic polynomial and $t\in \mbZ$. The Diophantine equation \begin{align}
F\left(\dfrac{x_1^2+x_2^4+x_3^4+2x_1x_2^2+2x_1x_3^2}{x_1x_2^2x_3^2}\right)=F(t) \label{F124}
	\end{align} has positive integer solutions if and only if $F(t)=F(7)$. Under this case, all the positive integer solutions can be obtained from the initial solution $(1,1,1)$ by finite cluster mutations of $\mcA_L$.
\end{proposition}
\begin{proof}
	Let $G(X)=F(X)-F(t)$ and it is a monic polynomial over $\mbZ$. Then, the \Cref{F14} is equivalent to 
	\begin{align}
G\left(\dfrac{x_1^2+x_2^4+x_3^4+2x_1x_2^2+2x_1x_3^2}{x_1x_2^2x_3^2}\right)=0.\label{G124}
	\end{align} According to \Cref{Atiyah}, the \Cref{G124} has positive integer solutions if and only if \begin{align}\frac{x_1^2+x_2^4+x_3^4+2x_1x_2^2+2x_1x_3^2}{x_1x_2^2x_3^2}\in \mbZ_{\geq 0}.\notag\end{align} Then, by \Cref{analogue}, we have \begin{align}\frac{x_1^2+x_2^4+x_3^4+2x_1x_2^2+2x_1x_3^2}{x_1x_2^2x_3^2}=7,\notag\end{align} which implies that $F(t)=F(7)$. Hence, based on \cite[Theorem 2.6]{Lam16}, all the positive integer solutions can be generated by the initial solution $(1,1,1)$ through finite cluster mutations of $\mcA_{P}$.
\end{proof}
\begin{remark}
	Finally, we conjectured that the Laurent mutation invariants of rank $3$ cluster algebras with irreducible sign-equivalent exchange matrices are generated by \eqref{Lampe mutation invariant} or \eqref{Lampe mutation invariant 3}.
\end{remark}

	
\subsection*{Acknowledgements}
The authors would like to express heartfelt thanks to Yu Ye and Zhe Sun for their support and help. The authors are also grateful to Tomoki Nakanishi and Ryota Akagi for their useful comments and suggestions. Z. Chen is supported by the China Scholarship Council (Grant No. 202406340022) and National Natural Science Foundation of China (Grant No. 124B2003).


\newpage
\begin{thebibliography}{99}
\newcommand{\au}[1]{\textrm{#1},}
\newcommand{\ti}[1]{\textrm{#1},}
\newcommand{\jo}[1]{\textit{#1}}
\newcommand{\vo}[1]{\textbf{#1}}
\newcommand{\yr}[1]{(#1)}
\newcommand{\pp}[2]{#1--#2.}
\newcommand{\arxiv}[1]{\href{http://arxiv.org/abs/#1}{arXiv:#1}}
\bibitem[Aig13]{Aig13}
\au{M. Aigner}
\ti{Markov's theorem and 100 years of the uniqueness conjecture: a mathematical journal from irrational numbers to perfect matchings}
\jo{Springer}
\yr{2013}.
\bibitem[AM69]{AM69}
\au{M. F. Atiyah, I. G. Macdonald} \ti{Introduction to Commutative Algebra} 
\jo{Addison-Wesley Publishing Company} \yr{1969}.
\bibitem[BBH11]{BBH11}
\au{A. Beineke, T. Br{\"u}stle, L. Hille}
\ti{Cluster-Cyclic Quivers with Three Vertices and the Markov Equation}
\jo{Algebr. Represent. Theory}
\vo{14} \yr{2011}, no. 1, \pp{97}{112}.
\bibitem[BL24]{BL24}
\au{L. Bao, F. Li}
\ti{A study on Diophantine equations via cluster theory}
\jo{J. Algebra}
\vo{639} \yr{2024}, \pp{99}{119}
\bibitem[BIRS09]{BIRS09}
\au{A.~Buan, O.~Iyama, I.~Reiten, J.~Scott}
\ti{Cluster structures for 2-Calabi-Yau categories and unipotent groups}
\jo{Compos. Math.}
\vo{145} \yr{2009}, \pp{1035}{1079}
\bibitem[BMRRT06]{BMRRT06}
\au{A. Buan, R. Marsh, M. Reineke, I. Reiten, G. Todorov}
\ti{Tilting theory and cluster combinatorics}
\jo{Adv. Math.}
\vo{204} \yr{2006}, \pp{572}{612}
\bibitem[CL24]{CL24}
\au{Z.~Chen, Z.~Li}
\ti{Mutation invariants of cluster algebras of rank 2}
\jo{J. Algebra}
\vo{654} \yr{2024}, \pp{25}{58}
\bibitem[DWZ08]{DWZ08}
\au{H. Derksen, J. Weyman, A. Zelevinsky} 
\ti{Quivers with potentials and their representations I: Mutations}
\jo{Selecta Mathematica} \vo{14} \yr{2008}, \pp{59}{119}
\bibitem[DWZ10]{DWZ10}
\au{H. Derksen, J. Weyman, A. Zelevinsky} 
\ti{Quivers with potentials and their representations II: applications to cluster algebras}
\jo{J. Amer. Math. Soc.} \vo{23(3)} \yr{2010}, \pp{749}{790}
\bibitem[FG09]{FG09}
\au{V.~Fock, A.~Goncharov}
\ti{Cluster ensembles, quantization and the dilogarithm}
\jo{Ann. Sci. Ecole Normale. Sup.}
\vo{42} \yr{2009}, \pp{865}{930}
\bibitem[FG19]{FG19} S. Fujiwara, Y. Gyoda, Duality between final-seed and initial-seed mutations in cluster algebras, \jo{SIGMA} \vo{15} (2019), 24 pages.
\bibitem[FR05]{FR05}
\au{S. Fomin, N. Reading}
\ti{Generalized cluster complexes and Coxeter combinatorics}
\jo{Int. Math. Res. Not.}
 \yr{2005}, no. 44, \pp{2709}{2757}
\bibitem[FST1]{FST1}
\au{A. Felikson, S. Shapiro, P. Tumarkin}
\ti{Skew-symmetric cluster algebras of finite mutation type}
\jo{J. Eur. Math. Soc.}
 \vo{14} \yr{2012}, no. 4, \pp{1135}{1180}
\bibitem[FST2]{FST2}
\au{A. Felikson, S. Shapiro, P. Tumarkin}
\ti{Cluster algebras of finite mutation type via
unfoldings}
\jo{Int. Math. Res. Not.}
\yr{2012}, no. 8, \pp{1768}{1804}
\bibitem[FZ02]{FZ02}
\au{S. Fomin, A. Zelevinsky}
\ti{Cluster Algebra I: Foundations}
\jo{J. Amer. Math. Soc.}
\vo{15} \yr{2002}, \pp{497}{529}
\bibitem[FZ03]{FZ03}
\au{S. Fomin, A. Zelevinsky}
\ti{Cluster Algebra II: Finite type classification}
\jo{Invent. Math.}
\vo{154} \yr{2003}, \pp{112}{164}
\bibitem[FZ04]{FZ04}
\au{S. Fomin, A. Zelevinsky}
\ti{Cluster algebras: Notes for the CDM-03 conference}
\jo{in: CDM 2003: Current Developments in Mathematics, International Press} \yr{2004}.
\bibitem[FZ07]{FZ07}
\au{S. Fomin, A. Zelevinsky}
\ti{Cluster Algebra IV: Coefficients}
\jo{Comp. Math.}
\vo{143} \yr{2007}, \pp{63}{121}
\bibitem[Gan98]{Gan98} \au{F. Gantmacher} \ti{The Theory of Matrices}, AMS Chelsea, American Mathematical Society: Providence, Rhode Island, \yr{1998}, Volume II.
\bibitem[GHKK18]{GHKK18}
\au{M. Gross, P. Hacking, S. Keel, M. Kontsevich}
\ti{Canonical bases for cluster algebras}
\jo{J. Amer. Math. Soc.}
\vo{31}, \yr{2018}, \pp{497}{608}
\bibitem[GM23]{GM23}
\au{Y. Gyoda, K. Matsushita}
\ti{Generalization of Markov Diophantine equation via generalized cluster algebra}
\jo{Electron. J. Comb.}
\vo{30}(4), \yr{2023}, P4.10.
\bibitem[GSV03]{GSV03}
\au{M. Gekhtman, M. Shapiro, A. Vainshtein}
\ti{Cluster algebras and Poisson geometry}
\jo{Mosc. Math. J.}
\vo{3} \yr{2003}, \pp{899}{934}
\bibitem[GSV12]{GSV12}
\au{M. Gekhtman, M. Shapiro, A. Vainshtein}
\ti{Poisson geometry of directed networks in an annulus}
\jo{J. Europ. Math. Soc.}
\vo{14} \yr{2012}, \pp{541}{570}
\bibitem[Hua22]{Hua22}
\au{M. Huang}
\ti{On the monotonicity of the generalized Markov numbers}
\arxiv{2204.11443}.
\bibitem[HL10]{HL10}
\au{D. Hernandez, B. Leclerc}
\ti{Cluster algebras and quantum affine algebras}
\jo{Duke Math. J.}
\vo{154(2)} \yr{2010}, \pp{265}{341}
\bibitem[KY11]{KY11}
\au{B. Keller, D. Yang}
\ti{Derived equivalences from mutations of quivers with potential}
\jo{Adv. Math.}
\vo{226(3)} \yr{2011}, \pp{2118}{2168}
\bibitem[KNS11]{KNS11}
\au{A. Kuniba, T. Nakanishi, J. Suzuki}
\ti{T-systems and Y-systems in integrable systems}
\jo{J. Phys. A: Math. Theor.}
\vo{44} \yr{2011}, 103001.
\bibitem[LLZ14]{LLZ14}  K. Lee, L. Li, A. Zelevinsky, Greedy elements in rank 2 cluster algebras, {\it Selecta Mathematica} {\bf 20} (2014), no. 1, 57-82.
\bibitem[LLRS23]{LLRS23}
\au{K. Lee, L. Li, M. Rabideau, R. Schiffler}
\ti{On the ordering of the Markov numbers}
\jo{Adv. Appl. Math.}
\vo{143} \yr{2023}, 102453.
\bibitem[Lam16]{Lam16}
\au{P. Lampe}
\ti{Diophantine equations via cluster transformations}
\jo{J. Algebra}
\vo{462} \yr{2016}, \pp{320}{337}
\bibitem[Mul13]{Mul13}
\au{G. Muller}
\ti{Locally acyclic cluster algebras}
\jo{Adv. Math.}
\vo{233} \yr{2013}, \pp{207}{247}
\bibitem[Mul14]{Mul14}
\au{G. Muller}
\ti{$\mcA=\mcU$ for locally acyclic cluster algebras}
\jo{SIGMA Symmetry Integrability Geom. Methods Appl.}
\vo{10} \yr{2014}, Paper 094.
\bibitem[Nak23]{Nak23}
\au{T. Nakanishi}
\ti{Cluster algebras and scattering diagrams}
\jo{MSJ Mem.} \vo{41} \yr{2023}, 279 pp; ISBN: 978-4-86497-105-8.
\bibitem[NZ12]{NZ12}
\au{T. Nakanishi, A. Zelevinsky}
\ti{On tropical dualities in cluster algebras}
\jo{Contemp. Math.} \vo{565} \yr{2012}, \pp{217}{226} 
\bibitem[Pro20]{Pro20}
\au{J. Propp}
\ti{The combinatorics of frieze patterns and Markoff numbers} \jo{Integers} \vo{20} \yr{2020}.
\bibitem[PZ12]{PZ12}
\au{X. Peng, J. Zhang}
\ti{Cluster algebras and markoff numbers}
\jo{CaMUS}
\vo{3} \yr{2012}, \pp{19}{26}
\bibitem[RS18]{RS18} N. Reading, S. Stella, Initial-seed recursions and dualities for $d$-vectors, \jo{Pacific J. Math} \vo{293} (2018), 179-206.
\end{thebibliography}
\end{document}